\documentclass[12pt]{article}

\usepackage{amsmath,amsthm,amsfonts,amssymb, color,xcolor,subcaption,graphicx} %,leftindex}

\usepackage{todonotes, dsfont}

\newtheorem{theorem}{Theorem}[section]

\newtheorem{proposition}[theorem]{Proposition}
\newtheorem{lemma}[theorem]{Lemma}
\newtheorem{definition}[theorem]{Definition}
\newtheorem{remark}[theorem]{Remark}

\def\cB{\mathcal{B}}
\def\cC{\mathcal{C}}
\def\cD{\mathcal{D}}

\def\cF{\mathcal{F}}
\def\cG{\mathcal{G}}
\def\cH{\mathcal{H}}

\def\cM{\mathcal{M}}
\def\cN{\mathcal{N}}
\def\cP{\mathcal{P}}

\def\cS{\mathcal{S}}

\def\cZ{\mathcal{Z}}

\def\bE{\mathbb{E}}

\def\bP{\mathbb{P}}
\def\bR{\mathbb{R}}

\def\bZ{\mathbb{Z}}

\newcommand{\les}{\lesssim}

\def\e{\varepsilon}
\def\k{\kappa}

\newcommand{\fm}{\mathfrak{m}}

\begin{document}

%% BEGIN Intro
\title{Gaussian fluctuations for hyperbolic Anderson model \\
with L\'evy colored noise}

\author{Raluca M. Balan\footnote{Corresponding author. University of Ottawa, Department of Mathematics and Statistics, 150 Louis Pasteur Private, Ottawa, Ontario, K1N 6N5, Canada. E-mail address: rbalan@uottawa.ca.} \footnote{Research supported by a grant from the Natural Sciences and Engineering Research Council of Canada.}
\and
William D. Stephenson \footnote{University of Ottawa, Department of Mathematics and Statistics, 150 Louis Pasteur Private, Ottawa, Ontario, K1N 6N5, Canada. E-mail address: wstep051@uottawa.ca.}
}

\date{February 25, 2026}
\maketitle

\begin{abstract}
\noindent In this article, we study the asymptotic behaviour of the spatial integral $F_R(t)$ of the solution to the hyperbolic Anderson model in dimension $d=1$, driven by the L\'evy colored noise introduced in \cite{BJ25}. We assume that the spatial coloration kernel of the noise is either integrable on $\bR$, or is the Riesz kernel of order $\alpha \in (0,1)$, and the L\'evy measure of the noise has finite moments of order $p$ and $2p$ for some $p \in (1,2]$. By applying a recent result of \cite{trauthwein25}, we prove that $F_R(t)/\sqrt{{\rm Var}\big(F_R(t)\big)}$ converges to the standard normal distribution as $R \to \infty$, and we give an estimate for the rate of this convergence in the Fortet-Mourier distance, the 1-Wasserstein distance, or the Kolmogorov distance. We also provide the corresponding functional limit result.
\end{abstract}

%60H15=SPDEs
%60G60=random fields
%60G51: Processes with independent increments; L\'evy processes

\vspace{1mm}


\pagebreak

\tableofcontents

\section{Introduction}

The theory of stochastic partial differential equations (SPDEs) has grown consistently in the last three decades, and now constitutes a central area of stochastic analysis. SPDEs model the evolution of random systems across multiple scales, ranging from microscopic particle dynamics to macroscopic continuum limits, and arise in diverse contexts, such as interface growth, disordered media, turbulence, and transport in random environments. They provide a rigorous mathematical framework for universality phenomena observed in physics. Motivated by these physical models and mathematical challenges, systematic theories for SPDEs began to emerge in the late 1980s and 1990s. Prominent developments include Walsh's \cite{walsh86} probabilistic approach based on martingale measures and random fields, and the semigroup approach introduced by Da Prato and Zabczyk \cite{DZ92}. Basic equations which have been studied are the stochastic heat equation and the stochastic wave equation, and their particular cases, the parabolic Anderson model (pAm), respectively the hyperbolic Anderson model (hAm). However, comprehensive solution theories for numerous other SPDEs
remained beyond the grasp of mathematical analysis, until the early 21st century. These challenges
stemmed from the singular nature of such equations, a singularity intricately linked to the irregularities inherent in the random data involved. A significant breakthrough arrived in 2014 when Hairer \cite{hairer14}
constructed a well-posedness framework for the Kardar-Parisi-Zhang (KPZ) equation, using concepts from Lyons' rough paths theory \cite{lyons98}. In dimension 1, the solution of the KPZ equation is directly related to the solution of (pAm
) via the Cole-Hopf transformation.

\medskip

Traditionally, SPDEs are perturbed by a space-time Gaussian white noise in dimension $d=1$, which is a zero-mean Gaussian process $\{B_t(\varphi);t> 0,\varphi \in L^2(\bR)\}$ with covariance:
\[
\bE[B_t(\varphi) B_s(\psi)]=(t \wedge s) \langle \varphi,\psi \rangle_{L^2(\bR)}.
\]
In higher dimensions, Dalang introduced in \cite{dalang99}, a spatially-homogeneous (or colored) noise, as a zero-mean Gaussian process $\{W_t(\varphi);t\geq 0,\varphi \in \cS(\bR)\}$ with covariance:
\begin{equation}
\label{def-cov-W}
\bE[W_t(\varphi) W_s(\psi)]=(t \wedge s) \int_{\bR^d} \int_{\bR^d}\varphi(x)\psi(y) f(x-y)dxdy,
\end{equation}
where $\cS(\bR^d)$ is the set of all rapidly decreasing functions on $\bR^d$, and $f:\bR^d \to [0,\infty]$ is the Fourier transform of a tempered measure $\mu$. This noise can also be defined as:
\begin{equation}
\label{Gauss-color}
W_t(\varphi)=B_t(\varphi*\k) \quad \mbox{for all $t>0$ and $\varphi \in \cS(\bR^d)$},
\end{equation}
where the kernel $\k$ is chosen such that $k * \widetilde{k}=f$, and $\widetilde{k}(x):=k(-x)$ for all $x \in \bR^d$.

\medskip

One of the problems which has received a lot of interest in the literature in the recent years is the study of the asymptotic behaviour as $R \to \infty$ of the spatial integral 
\[
F_R(t)=\int_{|x|<R} \big(u(t,x)-\bE[u(t,x)]\big)dx
\]
of the random-field solution $u$ of an SPDE. The major result is the {\em Quantitative Central Limit Theorem} (QCLT), which shows that $F_R(t)/\sqrt{{\rm Var}\big(F_R(t)\big)}$ converges in distribution as $R \to \infty$ to $Z \sim N(0,1)$, and gives an estimate for the rate of this convergence in the total variation distance. This was achieved for the first time in the seminal article \cite{HNV20} for the stochastic heat equation driven by space-time Gaussian white noise, by combining tools from Malliavin calculus with Stein's method for normal approximations. Since then, this method was extended to various other models, each with its own challenges. We refer the reader to \cite{BNZ,CKNP-delta,CKNP22,DNZ20,HNVZ,NSZ20,NXZ22,NZ20-1,NZ20-2,NZ22} for a sample of relevant references.

\medskip

In the recent article \cite{BZ24}, the QCLT problem was considered for the first time for an SPDE, namely (hAm),
driven by a L\'evy white noise, by using Poisson-Malliavin calculus tools (in this case, with respect to the underlying Poisson random measure of the noise), which have been developed for instance in \cite{Last16,LPS16}.  On the other hand, the recent article \cite{BJ25} has introduced a coloration procedure similar to \eqref{Gauss-color} for the L\'evy noise, and analyzed various SPDEs with this noise. Therefore, a natural question is whether the QCLT continues to hold for an SPDE driven by the L\'evy colored noise of \cite{BJ25}. We propose to address this question in the present article, by focusing on (hAm) in dimension $d=1$, as in \cite{BZ24}.

\medskip

More precisely, we consider the (hAm): 
\begin{align}
\label{HAM}
	\begin{cases}
		\dfrac{\partial^2 u}{\partial^2 t} (t,x)
		=  \dfrac{\partial^2 u}{\partial x^2} (t,x)+u(t,x) \dot{X}(t,x), \
		t>0, \ x \in \bR, \\
		u(0,x) = 1, \dfrac{\partial u}{\partial t} (0,x)=0, \quad x \in \bR,
	\end{cases}
\end{align}
where $X=\{X_t(\varphi);t>0,\varphi \in \cS(\bR)\}$ is the L\'evy colored noise, given by:
\[
X_t(\varphi)=L_t(\varphi *\k).
\]
with $L_t(\varphi)=L(1_{[0,t]} \varphi)$ for all $\varphi \in L^2(\bR)$, and $L=\{L(\varphi);\varphi \in L^2(\bR_{+} \times \bR)\}$ being the {\em L\'evy white noise}:
\[
L(\varphi)=\int_{\bR_{+} \times \bR \times \bR_0} \varphi(t,x)z \widehat{N}(dt,dx,dz) \quad \mbox{for all} \quad \varphi \in L^2(\bR_{+} \times \bR).
\]

Here $N$ is a Poisson random measure on $Z=\bR_{+}\times \bR \times \bR_0$ with intensity 
\[
\fm(dt,dx,dz)=dt dx \nu(dz),
\] 
defined on a complete probability space $(\Omega,\cF,\bP)$, and  $\widehat{N}(A)=N(A)-\fm(A)$ is the compensated version of $N$.
The space $\bR_0=\bR \verb2\2 \{0\}$ is equipped with the distance $d(x,y)=|x^{-1}-y^{-1}|$, so that the bounded subsets of $\bR_0$ are those that are bounded away from 0.
We assume that the measure $\nu$ satisfies the following condition:
\[
m_2:=\int_{\bR_0}|z|^2 \nu(dz)<\infty.
\]
Under this condition, the following isometry property holds: for any $\varphi,\psi \in L^2(\bR_{+}\times \bR)$,
\[
\bE[L(\varphi)L(\psi)]=m_2 \int_{\bR_+ \times \bR} \varphi(t,x)\psi(t,x)dtdx.
\]

We denote by $\cS'(\bR)$ the set of tempered distributions on $\bR$.
As in \cite{BJ25}, we assume that the kernel $\k$ satisfies the following assumption:

\medskip

\noindent {\bf Assumption A.}
$\k:\bR \to [0,\infty]$ is a continuous, symmetric, and tempered function such that:\\
(a) $\cF\k=h$ in $\cS'(\bR)$ is a tempered non-negative function, and $h^2$ is tempered;\\
(b) $f=\k* \widetilde{\k}$ is a continuous, symmetric, and tempered function ($\cF f=h^2$ in $\cS'(\bR))$;\\
(c)  $\k(x)<\infty$ for all $x\not=0$, and $f(x)<\infty$ for all $x\not=0$.

\medskip

We will impose Assumption A throughout this article.
Under this assumption, the noise $X$ is well-defined, since $\varphi*\k \in \L^2(\bR)$ for any $\varphi \in \cS(\bR)$. Moreover, for any $\varphi,\psi \in \cS(\bR)$,
\[
\int_{\bR} \int_{\bR}\varphi(x)\psi(y) f(x-y)dxdy=\int_{\bR} \cF \varphi(\xi) \overline{\cF \psi(\xi)} \mu(d\xi),
\]
where $\cF \varphi(\xi)=\int_{\bR}e^{-i \xi x} \varphi(x) dx$ is the Fourier transform of $\varphi$, and
\begin{equation}
\label{def-mu}
\mu(d\xi)=\frac{1}{2\pi} |\cF \k(\xi)|^2 d\xi.
\end{equation}

Since $d=1$, $\mu$ satisfies {\em Dalang's condition}: (see Remark 10.(b) of \cite{dalang99})
\[
\int_{\bR}\frac{1}{1+|\xi|^2}\mu(d\xi)<\infty.
\]

\medskip

The stochastic integral with respect to $X$ is defined using Walsh's theory \cite{walsh86}, 
since $X$ induces a martingale measure. This integral shares many properties with the integral defined in \cite{dalang99} for the Gaussian colored noise, except that its moments are estimated in a different way.

Let $G$ be the fundamental solution of the wave equation in dimension $d=1$, given by:
\begin{equation}
\label{def-G}
G_t(x)=\frac{1}{2}1_{\{|x|<t\}} \quad \mbox{for all $t>0$ and $x \in \bR$.}
\end{equation}

Inspired by D'Alembert formula, we introduce the following definition. 

\begin{definition}
{\rm A predictable process $\{u(t,x);t\geq 0,x\in \bR\}$ is called a {\em (mild) solution} of \eqref{HAM}
if it satisfies the integral equation:
\begin{equation}
\label{def-sol}
u(t,x)=1+\int_0^t \int_{\bR}G_{t-s}(x-y)u(s,y)X(ds,dy).
\end{equation}
}
\end{definition}

We recall that a random field $\{\Phi(t,x);t\geq 0,x\in \bR\}$ is {\em predictable} if it is measurable with respect to the predictable $\sigma$-field on $\Omega \times \bR_{+} \times \bR$, which is the minimal $\sigma$-field with respect to which all elementary processes are measurable. An {\em elementary process} is a linear combination of processes of the form
\[
\Phi(t,x)=Y 1_{(a,b]}(t) 1_{A}(x)
\]
where  $0\leq a<b$, $A \in \cB_b(\bR)$ and $Y$ is $\cF_a$-measurable. Predictable processes 
$\{\Phi(t,x,z);t\geq 0,x\in \bR,z\in \bR_0\}$ are defined similarly. Here $(\cF_t)_{t\geq 0}$ is the filtration induced by $N$, i.e.
\[
\cF_t=\sigma(\{N([0,s] \times A \times B); s \in [0,t],A\in \cB_b(\bR), B \in \cB_b(\bR_0)\}) \vee \cN,
\]
where $\cN$ is the class of $\bP$-null sets, and $\cB(\bR)$, $\cB_b(\bR_0)$ are the classes of bounded Borel subsets of $\bR$, respectively $\bR_0$.

\medskip

By Theorem 5.2 of \cite{BJ25}, equation \eqref{HAM} has a unique solution $u$, which satisfies:
\[
\sup_{(t,x)\in [0,T] \times \bR}\bE|u(t,x)|^2<\infty \quad \mbox{for all} \quad T>0.
\]

By Lemma 5.5 of \cite{BJ25}, the process $\{u(t,x)\}_{x\in \bR}$ is strictly stationary, for any $t>0$.
Moreover, by Theorem 5.6 of \cite{BJ25}, if for some $p\geq 2$, we have
\begin{equation}
\label{def-mp}
m_p:=\int_{\bR_0}|z|^p \nu(dz)<\infty,
\end{equation}
and
\begin{equation}
\label{def-Mp}
\cM_p(t):=\int_0^t \int_{\bR} \big|(G_{t-s}*\k)(x)\big|^p dx ds<\infty \quad \mbox{for all} \quad t>0,
\end{equation}
then
\begin{equation}
\label{mom-u}
K_p(T):=\sup_{(t,x)\in [0,T] \times \bR}\|u(t,x)\|_p<\infty \quad \mbox{for all} \quad T>0,
\end{equation}
where $\|\cdot\|_p$ denotes the norm in $L^p(\Omega)$. The goal of the present paper is to complement this analysis by showing that the {\em spatial average} of the (centered) solution:
\[
F_R(t)=\int_{-R}^{R}\big(u(t,x)-1 \big)dx.
\]
has asymptotic Gaussian fluctuations when $R \to \infty$, with a precise rate of convergence which depends on the noise.  We denote $\sigma_R^2(t)={\rm Var}(F_R(t))$.

\medskip

\medskip

Recall that the {\em Riesz kernel of order $\alpha \in (0,1)$} is given by:
\[
R_{1,\alpha}(x)=C_{1,\alpha}|x|^{-(1-\alpha)} \quad \mbox{where} \quad C_{1,\alpha}=\pi^{-1/2}2^{-\alpha}\frac{\Gamma(\frac{1-\alpha}{2})}{\Gamma(\frac{\alpha}{2})}.
\]
Moreover, $\cF R_{1,\alpha}(\xi)=|\xi|^{-\alpha}d\xi$, and $R_{1,\alpha}*R_{1,\beta}=R_{1,\alpha+\beta}$ for any $\alpha,\beta>0$ with $\alpha+\beta<1$.

\medskip

We introduce the following assumption.

\medskip

\noindent {\bf Assumption B.} 
(i) $\k\in L^1(\bR)$; or (ii) $\k=R_{1,\alpha/2}$ for some $\alpha \in (0,1)$.

%Examples of kernels $\k \in L^1(\bR)$ are the heat kernel or the Bessel kernel.

\begin{remark}
\label{rem-M}
{\rm Suppose that Assumption B holds. Applying Young's inequality if $\k \in L^1(\bR)$, and Theorem \ref{HLS} if $\k=R_{1,\alpha/2}$, we see that:
\[
\|G_t* \k\|_{L^p(\bR)} \leq
\left\{
\begin{array}{ll} 
\|\k\|_{L^1(\bR)} \|G_t\|_{L^p(\bR)} & \mbox{if $\k \in L^1(\bR)$ and $p\geq 1$} \\
A_{1,\alpha,p}\|G_t\|_{L^q(\bR)} & \mbox{if $\k=R_{1,\alpha/2}$, $\alpha \in (0,1)$, $p>\frac{2}{2-\alpha}$ and $\frac{1}{q}=\frac{1}{p}+\frac{\alpha}{2}$}
\end{array} \right.
\]
Since $\|G_t\|_{L^p(\bR)}^p=2^{1-p}t$ for any $p>0$, we infer that condition \eqref{def-Mp} holds for any $p\geq 2$. Therefore, \eqref{mom-u} holds for any $p \geq 2$ such that $m_p<\infty$.
}
\end{remark}

We recall the following definitions. Let $X$ and $Y$ be random variables defined on the same probability space.
The {\em $1$-Wasserstein distance} between $X$ and $Y$ is:
\[
d_{W}(X,Y)=\sup_{{\rm Lip}(h) \leq 1} \big|\bE[h(X)]-\bE[h(Y)]\big|, \quad
\mbox{where} \quad
{\rm Lip}(h):=\sup_{x \not=y}\frac{|h(x)-h(y)|}{|x-y|}.
\]
The {\em Fortet-Mourier distance} between $X$ and $Y$ is:
\[
d_{FM}(X,Y)=\sup_{\|h\|_{\infty}+{\rm Lip}(h) \leq 1} \big|\bE[h(X)]-\bE[h(Y)]\big|,
\]
The {\em Kolmogorov distance} between $X$ and $Y$ is:
\[
d_{K}(X,Y)=\sup_{x \in \bE} \big|P(X \leq x)-P(Y \leq x)\big|.
\]

We are now ready to state the main results of this article.

\begin{theorem}
\label{ergodic-th}
If Assumption B holds, then $\{u(t,x)\}_{x\in \bR}$ is ergodic, for any $t>0$. Consequently, by the mean ergodic theorem,
\[
\frac{1}{R}F_R(t) \to 0 \quad \mbox{a.s and in $L^2(\Omega)$}, \quad \mbox{as $R \to \infty$}.
\]
\end{theorem}

\begin{theorem}[Limiting covariance]
\label{cov-th}
If Assumption B holds, then
for any $t,s>0$,
\[
\lim_{R \to \infty}\frac{1}{R^{\beta}}\bE[F_R(t) F_R(s)]=K(t,s) \quad \mbox{is finite},
\] 
where 
\begin{equation}
\label{def-beta}
\beta:=\left\{
\begin{array}{ll} 1 & \mbox{if $\k \in L^1(\bR)$,} \\
\alpha+1 & \mbox{if $\k=R_{1,\alpha/2}$ for some $\alpha \in (0,1)$. }
\end{array} \right.
\end{equation}
In particular, $R^{-\beta} \sigma_R^2(t) \to K(t,t)$ as $R \to \infty$, for any $t>0$.
\end{theorem}

\begin{theorem}[Quantitative Central Limit Theorem]
\label{QCLT}
Suppose that Assumption B holds, and there exists $p \in (1,2]$ such that 
\begin{equation}
\label{mp-m2p}
m_{p}<\infty \quad \mbox{and} \quad m_{2p}<\infty.
\end{equation}

If $\k \in L^1(\bR)$, then
for any $t>0$,
\[
{\rm dist}\left( \frac{F_R(t)}{\sigma_R(t)},Z\right) \leq  C_t R^{-(1-\frac{1}{p})}
\]

If $\k=R_{1,\alpha/2}$ for some $\alpha \in (0,1)$, and $p>\frac{2}{2-\alpha}$, then
\[
{\rm dist}\left( \frac{F_R(t)}{\sigma_R(t)},Z\right) \leq
C_t R^{-\e}  \quad \mbox{for any $\varepsilon \in \big(0, 1-\frac{1}{p}-\frac{\alpha}{2}\big)$}.
\]

Here $C_t>0$ is a constant depending on $t$, and ${\rm dist}$ is the Fortet-Mourier distance, the 1-Wasserstein distance, or the Kolmogorov distance.
\end{theorem}

\begin{theorem}[Functional Central Limit Theorem]
\label{FCLT}
Under the hypotheses of Theorem \ref{QCLT}, for any $R>0$,
the process $\{F_R(t)\}_{t\geq 0}$ has a $\gamma$-H\"older continuous modification (denoted also $F_R$), for any $\gamma \in (0,\frac{\beta}{2})$, where $\beta$ is given by \eqref{def-beta}.
Moreover,
\[
\frac{1}{R^{\beta/2}}F_R(\cdot) \stackrel{d}{\to} \cG(\cdot) \quad \mbox{in $C[0,\infty)$ as $R \to \infty$},
\]
where $\{\cG(t)\}_{t\geq 0}$ is a zero-mean Gaussian process with covariance 
$\bE[\cG(t) \cG(s)]=K(t,s)$, and $K(t,s)$ given by Theorem \ref{cov-th}. Here $\stackrel{d}{\to}$ denotes the convergence in distribution, and $C[0,\infty)$ is equipped with the uniform convergence on compact sets.
\end{theorem}

The proofs of these theorems are based on a key estimate  which shows that the moments of the Malliavin derivative of the solution of \eqref{HAM} can be bounded, up to a constant, by a deterministic function (see relation \eqref{key-D} below). Similar estimates appear in all references dedicated to QCLT, where they play a crucial role. To prove this key estimate, we proceed as in \cite{BZ24}, developing a connection to (hAm) with delta initial velocity by using the form \eqref{def-G} of the fundamental solution $G$.

Theorem \ref{ergodic-th} follows from the key estimate \eqref{key-D} and a criterion which may be of independent interest (Lemma \ref{lem-key-Z} below). 
To prove Theorem \ref{cov-th} we show that the covariance of the solution of \eqref{HAM} coincides with the covariance of the solution of (hAm) driven by $\sqrt{m_2}W$, where $W$ is the Gaussian colored noise with covariance \eqref{def-cov-W}.

To prove Theorem \ref{QCLT}, we apply the recent result of Trauthwein \cite{trauthwein25}, which gives the optimal rates for the Wasserstein and Kolmogorov distances between $F/\sqrt{{\rm Var}(F)}$ and $Z \sim N(0,1)$, for a centered random variable $F$ with finite variance, which is Malliavin differentiable with respect to a compensated Poisson random measure. These estimates involve the first and second order Malliavin derivatives. Therefore, to  apply the result in \cite{trauthwein25}, we develop a similar key estimate for the second Malliavin derivative of the solution (see relation \eqref{key-D2} below). Different techniques are used for estimating the 7 quantities $\gamma_1,\ldots,\gamma_7$ which appear in the result of \cite{trauthwein25}, which rely on Young's inequality in the case when $k$ is integrable, respectively on Hardy-Littlewood-Sobolev inequality when $\k$ is the Riesz kernel. Finally, Theorem \ref{FCLT} follows by the classical method of finite dimensional convergence and tightness.

\begin{remark}
{\rm a) If $\k \in L^1(\bR)$, then $f=\k * \tilde{\k} \in L^1(\bR)$ by Young's inequality. In this case, the variance $\sigma_R^2(t)$ decays with rate $R$, as in the L\'evy white noise case and the Gaussian colored noise case with integrable kernel $f$ (see Theorem 1.2 of \cite{NZ22}). The decay rate $R^{-(1-\frac{1}{p})}$  in the QCLT (in the $d_{W}$, $d_{FM}$ or $d_K$ distances) is the same as in the L\'evy white noise case (see Theorem 1.1.(iii) of \cite{BZ24}), and depends on the parameter $p$ from \eqref{mp-m2p}. This can be explained since formally, the L\'evy white noise can be viewed as a L\'evy colored noise with integrable kernel $\k=\delta_0$. When $p=2$ (i.e. $m_4<\infty)$, this rate coincides with the rate $R^{-1/2}$ obtained in the case of the Gaussian colored noise case with integrable kernel $f$, for the QCLT in the total variation distance (see Theorem 1.2 of \cite{NZ22}).

\medskip

b) If $\k=R_{1,\alpha/2}$ for some $\alpha \in (0,1)$, then $f =R_{1,\alpha}$. In this case, using the parametrization $\alpha=2H-1$ with $H \in (\frac{1}{2},1)$, we see that the decay rate $R^{2H}$ of $\sigma_R^2(t)$ is the same as in the Gaussian colored noise case (see  Proposition 3.3 of \cite{DNZ20}). In this Gaussian case, Theorem 1.1 {\em ibid.} gives the rate $R^{-(1-H)}=R^{-\frac{1-\alpha}{2}}$ for the QCLT in the total variation distance, which is almost the same as the rate $R^{-\e}$ with $\e \in (0,\frac{1-\alpha}{2})$ obtained in Theorem \ref{QCLT} above when $p=2$ (i.e. $m_4<\infty$).
}
\end{remark}

This article is organized as follows. In Section \ref{section-prelim}, we include some preliminaries about Poisson-Malliavin calculus and moment inequalities with respect to noises $L$ and $X$. In Sections \ref{section-Mall1} and \ref{section-Mall2}, we prove some key estimates for the first, respectively the second,  Malliavin derivative of the solution. In Section \ref{section-proofs}, we present the proofs of Theorems \ref{ergodic-th}, \ref{cov-th}, \ref{QCLT} and \ref{FCLT}. Appendix \ref{appA} contains some auxiliary results about Poisson-Malliavin calculus. Appendix \ref{appB} presents some inequalities for Riesz potentials.

\medskip

We conclude the introduction with few words about the notation. We write $a\les b$ to indicate that $a \leq C b$ for some positive constant $C>0$ that does not depend on $(a,b)$. 
%We write $a \sim b$ if $a \les b$ and $b\les a$.
For any $p\geq 1$, we denote by $\|\cdot \|_p$ the norm in $L^p(\Omega)$.
%% END Intro

%% BEGIN Prelims
\section{Preliminaries}
\label{section-prelim}

In this section we include some preliminaries about Poisson-Malliavin calculus and some moment inequalities for the stochastic integral with respect to $X$.

\subsection{Malliavin calculus}

In this section, we include some basic material about Malliavin calculus with respect to the compensated Poisson random measure $\widehat{N}$. We refer the reader to Section 2.2 of \cite{BZ25} for more details. Additional facts are included in Appendix \ref{appA}.

\medskip

We consider the Hilbert space $\cH=L^2(Z,\cZ,\fm)$, where
\[
({\bf Z},\cZ,\fm)=\big(\bR_{+} \times \bR \times \bR_0,\ \cB(\bR_+)\otimes \cB(\bR) \otimes \cB(\bR_0), \ {\rm Leb} \times {\rm Leb} \times \nu\big).
\]

\medskip

$\bullet$
Any random variable $F \in L^2(\Omega)$ which is $\cF^N$-measurable has the {\em Poisson-chaos expansion}:
\begin{equation}
\label{Poisson-chaos}
F=\bE(F)+\sum_{n\geq 1}I_n(f_n), \quad \mbox{for some $f_n \in \cH^{\odot n}$},
\end{equation}
and the series is orthogonal in $L^2(\Omega)$. Here $I_n$ is the multiple integral with respect to $\widehat{N}$ and $\cH^{\odot n}$ is the set of symmetric functions in $\cH^{\otimes n}$.
For any $f \in \cH^{\otimes n}$,
\[
\bE[I_n(f)]=0 \quad \mbox{and} \quad \bE|I_n(f)|^2=n! \|\widetilde{f}\|_{\cH^{\otimes n}}^2,
\]
where $\widetilde{f}$ is the symmetrization of $f$:
\[
\widetilde{f}(\xi_1,\ldots,\xi_n)=\frac{1}{n!}\sum_{\rho \in S_n}f(\xi_{\rho(1)},\ldots,\xi_{\rho(n)}),
\]
$S_n$ being the set of all permutations of $1,\ldots,n$.
Moreover, $I_n(f)=I_n(\widetilde{f})$ for any $f \in \cH^{\otimes n}$.

\medskip

$\bullet$
For any random variable $F \in L^2(\Omega)$ with chaos expansion \eqref{Poisson-chaos}, we define the {\em Malliavin derivative} of $F$ by:
\[
D_{\xi}F=\sum_{n\geq 1}nI_{n-1}\big(f_n(\cdot,\xi)\big), \quad \mbox{for any $\xi \in {\bf Z}$},
\]
provided that provided that the series converges in $L^2(\Omega)$, that is 
\begin{equation}
\label{check-D1}
\bE\|DF\|_{\cH}^2=\sum_{n\geq 1}nn!  \|f_n\|_{\cH^{\otimes n}}^2 <\infty.
\end{equation}
 In this case, we write $F \in {\rm dom}(D)$.
We have the following {\em Poincar\'e inequality}:
\begin{equation}
\label{poincare}
{\rm Var}(F) \leq \bE\|DF\|_{\cH}^2 \quad \mbox{for any $F \in {\rm dom}(D)$.}
\end{equation}

The following chain rule holds: for any $F \in {\rm dom}(D)$,
\[
D_{\xi}\varphi(F) =\varphi(F+D_{\xi}F)-\varphi(F) \quad \mbox{for any $\xi \in {\bf Z}$}.
\]
Hence, if $\varphi$ is Lipschitz with Lipschitz constant ${\rm Lip}(\varphi)$, 
$|D_{\xi}\varphi(F)| \leq {\rm Lip}(\varphi)|D_{\xi}F|$, and
\begin{equation}
\label{chain1}
\|D_{\xi}\varphi(F)\|_p \leq {\rm Lip}(\varphi)\|D_{\xi}F\|_p \quad \mbox{for any $p>0$}.
\end{equation}

If $F \in {\rm Dom}(D)$ is such that $F$ is $\cF_t$-measurable for some $t \in \bR_{+}$, then 
\[
D_{r,x,z}F=0 \quad \mbox{for any $r>t$, $x \in \bR$ and $z \in \bR_0$}.
\]

$\bullet$
Similarly, we can define the second Malliavin derivative $D^2 F$ as follows:
If $F$ has the chaos expansion \eqref{Poisson-chaos}, we let
\begin{align*}
D^2_{\xi_1,\xi_2} F = D_{\xi_1} D_{\xi_2} F
= \sum_{n \geq 2} n(n-1) I_{n-2}\big( f_{n-2}(\cdot,\xi_1, \xi_2 )\big), \quad \mbox{for $(\xi_1,\xi_2) \in {\bf Z}^2$},
\end{align*}
provided that the series converges in $L^2(\Omega)$, that is
\begin{equation}
\label{check-D2}
\bE\|D^2F\|_{\cH^{\otimes 2}}^2=\sum_{n \geq 2} n(n-1) n! \| f_n \|^2_{\cH^{\otimes n}} < \infty.
\end{equation}
In this case, we write $F \in {\rm dom}(D^2)$.

\medskip

$\bullet$
 Let $\delta:{\rm Dom}(\delta) \to L^2(\bR)$ be the adjoint of $D$, whose domain ${\rm Dom}(\delta)$ is the set of $V \in L^2(\Omega;\cH)$ for which there exists a constant $C=C_V>0$ depending on $V$, such that 
\[
\big|\bE \langle DF,V \rangle_{\cH}\big| \leq C \|F\|_2 \quad \mbox{for any $F \in {\rm Dom}(D)$}.
\] 
We say that $\delta(V)$ is 
{\em the Skorohod integral} of $V$ with respect to $\widehat{N}$, and we write
\[
\delta(V)=\int_{\bR_{+}}\int_{\bR} \int_{\bR_0}V(t,x,z)\widehat{N}(\delta t, \delta x, \delta z).
\]
By duality, for any $V \in {\rm Dom}(\delta)$,
\[
\bE\langle DF,V \rangle_{\cH}=\bE[F \delta(V)] \quad \mbox{for any $F \in {\rm Dom}(D)$}.
\]

If $V \in L^2(\Omega;\cH)$ is predictable, then $V \in {\rm dom}(\delta)$ and $\delta(V)$ coincides with the It\^o integral with respect to $\widehat{N}$. (A process $\{V(\xi);\xi \in {\bf Z}\}$ is {\em predictable} if it is measurable with respect to the predictable $\sigma$-field on $\Omega \times {\bZ}$.)

\subsection{Moment inequalities}

In this section, we include some moment inequalities which are used in the sequel.

We recall Rosenthal's inequality for the stochastic integral with respect to $L$. 

\begin{proposition}[Proposition 3.1 of \cite{BN16}]
\label{ros-prop}
Let $\Phi=\{\Phi(t,x);t\in [0,T],x \in \bR\}$ be a predictable process such that $\Phi \in L^2(\Omega \times [0,T] \times \bR)$. If $p\geq 2$ is such that $m_p<\infty$,
then
\begin{equation}
\label{rosenthal}
\bE\left|\int_0^T \int_{\bR}\Phi(t,x) L(dt,dx)\right|^p \leq \cC_p \left\{ \bE \left( \int_0^T \int_{\bR} |\Phi(t,x)|^2 dx dt \right)^{p/2} + \bE \int_0^T \int_{\bR}|\Phi(t,x)|^p dxdt\right\},
\end{equation}
where $\cC_p=2^{p-1}B_p^p(m_2^{p/2} \vee m_p)$ and $B_p$ is the constant in Rosenthal's inequality.
\end{proposition}

The following result is a modified version of Theorem 3.7 of \cite{BJ25}, which does not require that $Z$ satisfies Hypothesis A ibid. (This result remains valid in any dimension $d\geq 1$.)

\begin{theorem}
\label{mom-SZ-th}
Let $p\geq 2$ be such that $m_p<\infty$, and $\{Z(t,x);t\in [0,T],x\in \bR\}$ be a predictable process such that 
\[
\sup_{(t,x)\in [0,T] \times \bR} \bE|Z(t,x)|^p<\infty.
\]
Let $S$ be a deterministic non-negative function on $[0,T]$ such that $S \in \cP_{+}$. Then $SZ \in \cP_{+}$  and for any $t \in [0,T]$,
\begin{align*}
& \bE\left| \int_0^t \int_{\bR}S(s,x)Z(s,x)X(ds,dx)\right|^p \leq \cC_p(t)\int_0^t \sup_{x\in \bR} \bE|Z(s,x)|^p \\
& \quad \quad \quad \left( \int_{\bR}\int_{\bR} S(s,y)S(s,z)f(y-z)dydz+\int_{\bR}\big|\big(S(s,\cdot)*\k\big)(x)\big|^p dx \right) ds,
\end{align*}
where $\cC_p(t)=\cC_p \max(\nu_t^{p/2-1},1)$, $\cC_p$ is the constant from Proposition \ref{ros-prop}, and 
\[
\nu_t=\int_0^t \int_{\bR}\int_{\bR} S(s,x)S(s,y)f(x-y)dxdyds.
\]
\end{theorem}

\begin{proof}
Using the fact that $S$ is non-negative, and the Cauchy-Schwarz inequality,
\begin{align*}
\|SZ\|_{+}^2&=m_2 \bE \int_0^T \int_{\bR} \int_{\bR} |S(t,x)S(t,y)Z(t,x)Z(t,y)|f(x-y)dxdydt\\
& =m_2\int_0^T \int_{\bR} \int_{\bR} S(t,x)S(t,y)\bE|Z(t,x)Z(t,y)|f(x-y)dxdydt \\
& \leq \sup_{(t,x) \in [0,T]\times \bR} \bE|Z(t,x)|^2 \|S\|_{+}<\infty.
\end{align*}
Hence, $SZ \in \cP_{+}$. By Lemma 3.1 of \cite{BJ25} (applied to $g=SZ$),
\[
M_t:=\int_0^t \int_{\bR}S(s,x)Z(s,x)X(ds,dx)=\int_0^t \int_{\bR}\big(SZ(s,\cdot)*\k\big)(x)L(ds,dx).
\]
By Proposition \ref{ros-prop},
\begin{align*}
\bE|M_t|^p & \leq \cC_p  \left\{ \bE \left( \int_0^t \int_{\bR} \big|\big((SZ)(s,\cdot)*\k\big)(x)\big|^2 dx ds \right)^{p/2} + \right. \\
& \quad \quad \quad \left. \bE \int_0^t \int_{\bR}\big|\big((SZ)(s,\cdot)*\k\big)(x)\big|^p dxds\right\}
 =:\cC_p(T_1+T_2).
\end{align*}
The terms $T_1$ and $T_2$ are now estimated as in the proof of Theorem 3.6 of \cite{BJ25} ({\em Case 1}).
\end{proof}
%% END Prelims

%% BEGIN MDsol
\section{Malliavin derivative of the solution}
\label{section-Mall1}

In this section, we provide an estimate for the $p$-th moment of the Malliavin derivative of the solution.

Let $u$ be the solution of \eqref{HAM}. By Lemma 3.1 of \cite{BJ25} and the fact that the integrand is predictable, we have:
\begin{align*}
u(t,x) &= 1+ \int_0^t \int_{\bR} \big(G_{t-s}(x-\cdot) * \k\big)(y) L(ds,dy) \\
&=1+\int_0^t \int_{\bR} \big(G_{t-s}(x-\cdot) * \k\big)(y) z \widehat{N}(\delta s,\delta y, \delta z).
\end{align*}

Theorem 7.3 of \cite{BJ25} shows that $u(t,x)$ has the Poisson-chaos expansion:
\begin{equation}
\label{chaos-u}
u(t,x)=1+\sum_{n\geq 1}I_n\big(f_n^*(\cdot,t,x)\big),
\end{equation}
where the kernel $f_n^*(\cdot,t,x)$ is given by:
\begin{align*}
f_n^*(t_1,x_1,z_1,\ldots,t_n,x_n,z_n,t,x)& =\big(f_n(t_1,\cdot,\ldots,t_n,\cdot,t,x)* \k^{\otimes n}\big)(x_1,\ldots,x_n)z_1 \ldots z_n 
\end{align*}
with $\k^{\otimes n}(x_1,\ldots,x_n)=k(x_1)\ldots k(x_n)$ and
\begin{equation}
\label{fn-defn}
f_n(t_1,x_1,\ldots,t_n,x_n,t,x)=G_{t-t_n}(x-x_n)\ldots G_{t_2-t_1}(x_2-x_1)1_{\{0<t_1<\ldots<t_n<t\}}.
\end{equation}

We let $f_0^*=f_0=1$ and $I_0(x)=x$ for all $x \in \bR$.

\medskip

Note that the symmetrization 
$\widetilde{f}_n^*(\cdot,t,x)$ of $f_n^*(\cdot,t,x)$ is given by:
\[
\widetilde{f}_n^*(t_1,x_1,z_1,\ldots,t_n,x_n,z_n,t,x)= \big(\widetilde{f}_n(t_1,\cdot,\ldots,t_n,\cdot,t,x) * \k^{\otimes}\big)(x_1,\ldots,x_n) \prod_{i=1}^{n} z_i.
\]

Applying Plancherel theorem, we obtain:
\begin{align}
\nonumber
\|\widetilde{f}_n^*(\cdot,t,x)\|_{\cH^{\otimes n}}^2 &=\frac{m_2^n}{(2\pi)^n} \int_{[0,t]^n}\int_{\bR^n}\big|\cF \big(\widetilde{f}_n(t_1,\cdot,\ldots,t_n,\cdot,t,x)*\k^{\otimes n}\big)(\pmb{\xi})\big|^2 d\pmb{\xi} d\pmb{t}\\
\nonumber
& = \frac{m_2^n}{(2\pi)^n} \int_{[0,t]^n}\int_{\bR^n} \big|\cF \widetilde{f}_n(t_1,\cdot,\ldots,t_n,\cdot,t,x)(\pmb{\xi}) \big|^2 \prod_{i=1}^{n}\big|\cF \k(\xi_i)\big|^2 d\pmb{\xi}d\pmb{t}\\
\label{Poi-Gauss}
& = m_2^n \, \|\widetilde{f}_n(\cdot,t,x)\|_{\cH_0^{\otimes n}}^2,
\end{align}
where $\pmb{\xi}=(\xi_1,\ldots,\xi_n)$, $\pmb{t}=(t_1,\ldots,t_n)$, 
\[
\langle \varphi,\psi \rangle_{\cH_0}=\int_{0}^{\infty}\int_{\bR^d}\varphi(t,x)\psi(t,y)f(x-y)dxdydt 
=\int_0^{\infty}\int_{\bR}\cF \varphi(t,\cdot)(\xi)\overline{\cF \psi(t,\cdot)(\xi)}\mu(d\xi)dt,
\]
and we recall that the measure $\mu$ is given by \eqref{def-mu}. Therefore,
\begin{align*}
& \bE|u(t,x)|^2=\sum_{n\geq 1}n!\| \widetilde{f}_n^*(\cdot,t,x)\|_{\cH^{\otimes n}}^2=
\sum_{n\geq 1}n!\, m_2^n \|\widetilde{f}_n(\cdot,t,x)\|_{\cH_0^{\otimes n}}^2=\bE|U(t,x)|^2,
\end{align*}
where $U=\{U(t,x);t\geq 0,x\in \bR\}$ is the solution of the hyperbolic Anderson model:
\begin{align}
\label{HAM-W}
	\begin{cases}
		\dfrac{\partial^2 U}{\partial^2 t} (t,x)
		=  \dfrac{\partial^2 U}{\partial x^2} (t,x)+\sqrt{m_2}\,U(t,x) \dot{W}(t,x), \
		t>0, \ x \in \bR, \\
		U(0,x) = 1, \quad \quad \dfrac{\partial U}{\partial t} (0,x)=0, \quad x \in \bR,
	\end{cases}
\end{align}
driven by the colored Gaussian noise $W$  with covariance \eqref{def-cov-W}. Note that $W$ can be extended to an isonormal Gaussian process $\{W(\varphi)\}_{\varphi \in \cH_0 }$, where $\cH_0$ is the completion of $\cD(\bR_{+} \times \bR)$ with respect to $\langle \cdot,\cdot \rangle_{\cH_0}$. The solution $U(t,x)$ has the Wiener-chaos expansion:
\begin{equation}
\label{chaos-UW}
U(t,x)=1+\sum_{n\geq 1}m_2^{n/2}I_n^W\big( f_n(\cdot,t,x)\big),
\end{equation}
where $I_n^W$ denotes the multiple Wiener integral of order $n$, with respect to $W$.

\medskip

Using the connection to this Gaussian-driven model, we now derive the following result.

\begin{proposition}
\label{u-Mall1}
For any $t>0$ and $x \in \bR$, $u(t,x) \in {\rm dom}(D)$.
\end{proposition}

\begin{proof}
It is known that
$U(t,x)$ is Malliavin differentiable with respect to $W$ (see e.g. Proposition 7.1 of \cite{sanz05}). Based on the Wiener-chaos expansion \eqref{chaos-UW}, we deduce that the Malliavin derivative of $U(t,x)$ (with respect to $W$) is given by:
\[
D_{r,y}^W U(t,x)=1+\sum_{n\geq 1}m_2^{n/2}nI_{n-1}^W\big(\widetilde{f_n}(\cdot,r,y,t,x)\big), \quad r \in [0,t],y \in \bR.
\]
Using relation \eqref{Poi-Gauss}, we infer that
\[
\sum_{n\geq 1}n n! \|\widetilde{f_n^*}(\cdot,t,x)\|_{\cH^{\otimes n}}^2 =\sum_{n\geq 1}n n! \, m_2^{n} \|\widetilde{f_n}(\cdot,t,x)\|_{\cH_0^{\otimes n}}^2 =\bE\|D^W U(t,x)\|_{\cH_0}^2<\infty.
\]
The fact that $u(t,x)\in {\rm dom}(D)$ now follows; see \eqref{check-D1}.
\end{proof}

The Malliavin derivative of the solution has the chaos expansion:
\begin{equation}
\label{chaos-Du}
D_{r,y,z}u(t,x)=\sum_{n\geq 1}n I_{n-1}\big(\widetilde{f}_n^*(\cdot,r,y,z,t,x)\big),
\end{equation}
where $\widetilde{f}_n^*(\cdot,r,y,z,t,x)$ is the symmetrization of $f_n^*(\cdot,r,y,z,t,x)$. Note that
\begin{equation}
\label{fn*1}
\widetilde{f}_n^*(\cdot,r,y,z,t,x)=\frac{1}{n}\sum_{j=1}^{n}h_j^{*(n)}(\cdot,r,y,z,t,x),
\end{equation}
where $h_j^{*(n)}(\cdot,r,y,z,t,x)$ is the symmetrization of the function $f_j^{*(n)}(\cdot,r,y,z,t,x)$:
\[
h_j^{*(n)}(\xi_1,\ldots,\xi_{n-1},r,y,z,t,x)=\frac{1}{(n-1)!} \sum_{\rho \in S_{n-1}}f_j^{*(n)}(\xi_{\rho(1)},\ldots,\xi_{\rho(n-1)},r,y,z,t,x).
\]

The function $f_j^{*(n)}(\cdot,r,y,z,t,x)$ is obtained by placing $(r,y,z)$ in the $j$-th position among the arguments of $f_n^{*}(\cdot,t,x)$. More precisely, letting $\xi_{\ell}=(t_{\ell},x_{\ell},z_{\ell})$ for $\ell \leq n-1$, we define for any $j=1,\ldots,n$,
\begin{align*}
&  f_j^{*(n)}(\xi_1,\ldots,\xi_{n-1},r,y,z,t,x) :=f_n^{*}(\xi_1,\ldots,\xi_{j-1},,r,y,z,\xi_j,\ldots,
\xi_{n-1},t,x) \\
& = z \prod_{i=1}^{n-1}z_{i} \int_{\bR} \k(y-y') \left( \int_{\bR^{j-1}} G_{r-t_{j-1}}(y'-y_{j-1}) \ldots G_{t_2-t_1}(y_2-y_1) \prod_{i=1}^{j-1}\k(x_{i}-y_{i}) )dy_1 \ldots dy_{j-1} \right)\\
& \left( \int_{\bR^{n-j}} G_{t-t_{n-1}}(x-y_{n-1}) \ldots G_{t_{j+1}-t_j}(y_{j+1}-y_j)G_{t_j-r}(y_j-y') \prod_{i=j}^{n-1}\k(x_{i}-y_{i}) dy_j \ldots dy_{n-1}\right) dy'.
\end{align*}

It is useful to express this function using the following notation: for any $n\geq 1$,
\begin{equation}
\label{def-gn*}
g_n^*(t_1,x_1,z_1\ldots,t_n,x_n,z_n,r,y,z,t,x)= \big(g_n(t_1,\cdot,\ldots,t_n,\cdot,r,y,t,x)* \k^{\otimes n}\big)(x_1,\ldots,x_n) \prod_{i=1}^n z_{i} z,
\end{equation}
where
\[
g_n(t_1,x_1,\ldots,t_n,x_n,r,y,t,x)=G_{t-t_n}(x-x_n)\ldots G_{t_2-t_1}(x_2-x_1)G_{t_1-r}(x_1-y) 1_{\{r<t_1<\ldots<t_n<t\}}.
\]
We let $g_0(r,y,t,x)=G_{t-r}(x-y)$ and $g_0^*(r,y,z,t,x)=G_{t-r}(x-y)z$. With this notation, 
\begin{align*}
& f_j^{*(n)}(\xi_1,\ldots,\xi_{n-1},r,y,z,t,x)= \\
& \quad \int_{\bR}f_{j-1}^* (\xi_1,\ldots,\xi_{j-1},r,y') g_{n-j}^*(\xi_j,\ldots,\xi_{n-1},r,y',z,t,x)\k(y-y')dy',
\end{align*}
for any $j=1,\ldots,n$ and $n\geq 1$.

\medskip

We are now ready to state our key estimate for the Malliavin derivative of the solution $u$, which shows that the first term in the chaos expansion \eqref{chaos-Du} dominates the other terms. This estimate is the analogue of relation (3.25) of \cite{BZ24} (in the case of the L\'evy white noise) for the L\'evy colored noise.

\begin{theorem}
\label{key-thD}
Let $p\geq 2$ be such that $m_p<\infty$. If \eqref{def-Mp} holds,
then for any $0\leq r\leq t \leq T$, $x \in \bR$, $y \in \bR$ and $z \in \bR_0$,
\begin{equation}
\label{key-D}
\|D_{r,y,z}u(t,x)\|_p \leq C'_{T,p,\nu,\k}|z| \int_{\bR}G_{t-r}(x-y')\k(y-y')dy',
\end{equation}
where $C_{T,p,\nu,\k}'=2K_p(T) C_{T,p,\nu,\k}$ with $K_p(T)$ and $C_{T,p,\nu,\k}$ are given by
\eqref{mom-u}, respectively \eqref{mom-v}. In particular, under Assumption B, relation \eqref{key-D} holds for any $p\geq 2$ such that $m_p<\infty$.
\end{theorem}

\begin{proof} Using \eqref{chaos-Du} and \eqref{fn*1}, we have:
\begin{align*}
D_{r,y,z}u(t,x)&=\sum_{n\geq 1} \sum_{j=1}^{n}I_{n-1}\big(f_j^{*(n)}(\cdot,r,y,z,t,x)\big)=
\sum_{j\geq 1} \sum_{n\geq j}I_{n-1}\big(f_j^{*(n)}(\cdot,r,y,z,t,x)\big).
\end{align*}

By the stochastic Fubini theorem given by Lemma 2.6 of \cite{BZ24}, followed by the product formula given by Lemma \ref{prod-lem} below, 
\begin{align*}
I_{n-1}\big(f_j^{*(n)}(\cdot,r,y,z,t,x)\big) &=
\int_{\bR}I_{n-1}\big(f_{j-1}^*(\cdot,r,y')\otimes g_{n-j}^*(\cdot,r,y',z,t,x)\big)\k(y-y')dy'\\
&=\int_{\bR}I_{j-1}\big(f_{j-1}^*(\cdot,r,y')\big) I_{n-j}\big( g_{n-j}^*(\cdot,r,y',z,t,x)\big)\k(y-y')dy'.
\end{align*}

Therefore,
\begin{align*}
D_{r,y,z} u(t,x) &= \int_{\bR} \left(\sum_{j\geq 1} I_{j-1}\big(f_{j-1}^*(\cdot,r,y') \big) \right) \left(\sum_{n\geq j} I_{n-j}\big( g_{n-j}^*(\cdot,r,y',z,t,x)\big) \right) \k(y-y')dy'.
\end{align*}

The first series is the chaos expansion \eqref{chaos-u} of $u(r,y')$. The second series is also a chaos expansion, namely of the solution $v^{(r,y',z)}$ of equation \eqref{HAM-delta} below (see \eqref{chaos-v}). It follows that:
\[
D_{r,y,z} u(t,x)=\int_{\bR} u(r,y')v^{(r,y',z)}(t,x)\k(y-y')dy'.
\]
Note that the variables $u(r,y')$ and $v^{(r,y',z)}(t,x)$ are independent: $u(r,y')$ is $\cF_r$-measurable, while $v^{(r,y',z)}(t,x)$ is $\cF_{r,t}$-measurable, where
\begin{equation}
\label{def-Frt}
\cF_{r,t}:=\sigma(\{N(A \times B \times C); A \subset (r,t], B \in \cB_b(\bR),C \in \cB_b(\bR_0)\}),
\end{equation}
and $\cF_r$ and $\cF_{r,t}$ are independent since the sets $[0,r]$ and $(r,t]$ are disjoint. Therefore, using Minkowski inequality, followed by \eqref{mom-u} and \eqref{2Gv-1}, we have:
\begin{align*}
\|D_{r,y,z} u(t,x)\|_p & \leq \int_{\bR} \|u(r,y')\|_p\|v^{(r,y',z)}(t,x)\|_p\k(y-y')dy'\\
& \leq 2 K_p(T) C_{T,p,\nu,\k}|z|\int_{\bR}G_{t-r}(x-y')\k(y-y')dy'.
\end{align*}
Relation \eqref{key-D} follows. The final statement follows by Remark \ref{rem-M}.
\end{proof}

\begin{lemma}
\label{prod-lem} 
 If $f \in \cH^{\otimes n}$ contains the indicator of $\{0 < t_1 < \cdots < t_n < r \}$ and $g \in \cH^{\otimes m}$ contains the indicator of $\{r < t_1 < \cdots < t_m \}$ for some integers $n,m \geq 1$, then
\[
I_{n+m}(f\otimes g)=I_{n}(f) I_{m}(g).
\]
\end{lemma}

\begin{proof} By Lemma \ref{prod-lem2}, it suffices to show that for any $k=1,\ldots,n\wedge m$,
\[
\tilde{f} \star_k^1 \tilde{g}=0 \quad \mbox{and} \quad \tilde{f} \star_k^0 \tilde{g}=0.
\]

We first show that $\tilde{f} \star_k^1 \tilde{g}=0$ for all $k=1,\ldots,n\wedge m$. We denote $\zeta_i=(\alpha_i,\beta_i,\gamma_i)$, $\xi_i=(t_i,x_i,z_i)$ and  $\xi_i'=(t_i',x_i',z_i')$. Then,
\begin{align*}
& (\tilde{f} \star_k^1 \tilde{g})(\zeta_1,\ldots,\zeta_{k-1},\xi_1,\ldots,\xi_{n-k},\xi_1',\ldots,\xi_{m-k}')\\
& = \int_{Z} \tilde{f}(a,b,c,\zeta_1,\ldots,\zeta_{k-1},\xi_1,\ldots,\xi_{n-k})
\tilde{g}(a,b,c,\zeta_1,\ldots,\zeta_{k-1},\xi_1',\ldots,\xi_{m-k}')da db \nu(dc)=0.
\end{align*}
This is because $f(a,b,c,\zeta_1,\ldots,\zeta_{k-1},\xi_1,\ldots,\xi_{n-k})$ contains the indicator of
\[
\{0<a<\alpha_1<\ldots< \alpha_{k-1}<t_1<\ldots<t_{n-k}<r\},
\]
so $\tilde{f}(a,b,c,\zeta_1,\ldots,\zeta_{k-1},\xi_1,\ldots,\xi_{n-k})$ contains the indicator of $\{0 < a < r\}$. Likewise, $\tilde{g}(a,b,c,\zeta_1,\ldots,\zeta_{k-1},\xi_1',\ldots,\xi_{m-k}')$ contains the indicator of $\{r<a\}$ and these sets are disjoint.

For the modified contraction corresponding to $\ell=0$, we have:
\begin{align*}
& (\tilde{f} \star_k^0 \tilde{g})(\zeta_1,\ldots,\zeta_{k},\xi_1,\ldots,\xi_{n-k},\xi_1',\ldots,\xi_{m-k}')=\\
& \quad  \tilde{f}(\zeta_1,\ldots,\zeta_{k},\xi_1,\ldots,\xi_{n-k})
\tilde{g}(\zeta_1,\ldots,\zeta_{k},\xi_1',\ldots,\xi_{m-k}')=0,
\end{align*}
since $\tilde{f}(\zeta_1,\ldots,\zeta_{k},\xi_1,\ldots,\xi_{n-k})$ contains the indicator of
$\{0<\alpha_1<r\}$ and $\tilde{g}(\zeta_1,\ldots,\zeta_{k},\linebreak \xi_1',\ldots,\xi_{m-k}')$ contains the indicator of $\{r<\alpha_1\}$.
\end{proof}

In the proof of Theorem \ref{key-thD}, we used the chaos expansion of the solution $v^{(r,y,z)}$ of the following equation:
\begin{align}
\label{HAM-delta}
	\begin{cases}
		\dfrac{\partial^2 v}{\partial^2 t} (t,x)
		=  \dfrac{\partial^2 v}{\partial x^2} (t,x)+v(t,x) \dot{X}(t,x), \
		t>r, \ x \in \bR, \\
		v(r,x) = 0, \quad \dfrac{\partial v}{\partial t} (r,x)=z\delta_y, \quad x \in \bR,
	\end{cases}
\end{align}

We say that $v$ is a (mild) {\bf solution} of \eqref{HAM-delta} if it satisfies the integral equation:
\begin{equation}
\label{HAM-delta-int}
v(t,x)=G_{t-r}(x-y)z+\int_r^t \int_{\bR}G_{t-s}(x-y)v(s,y)X(ds,dy),
\end{equation}
for any $t \geq r$ and $x \in \bR$.
We now prove that \eqref{HAM-delta-int} has a unique solution, this solution has the desired chaos expansion, and its $p$-th moments can be bounded by a factor of $G_{t-r}(x-y)|z|$.

\begin{proposition}
\label{properties-v}
a) For any $(r,y,z)\in Z$, equation \eqref{HAM-delta} has a unique solution $v^{(r,y,z)}$. Let $p\geq 2$ be such that $m_p<\infty$. If \eqref{def-Mp} holds, then for any $T>0$,
\begin{equation}
\label{mom-v}
\sup_{r\leq t\leq T}\sup_{x,y\in \bR}\|v^{(r,y,z)}(t,x)\|_p \leq C_{T,p,\nu,\k}|z|,
\end{equation}
where $C_{T,p,\nu,\k}>0$ is a constant that depends on $(T,p,\nu,\k)$.

b) $v^{(r,y,z)}$ has the chaos expansion:
\begin{equation}
\label{chaos-v}
v^{(r,y,z)}(t,x)=zG_{t-r}(x-y)+\sum_{n\geq 1} I_{n-1}\big(g_n^*(\cdot,r,y,z,t,x) \big).
\end{equation}

c) For any $0\leq r\leq t$, $x \in \bR$, $y \in \bR$ and $z \in \bR_0$,
\begin{equation}
\label{2Gv}
v^{(r,y,z)}(t,x)=2G_{t-r}(x-y)v^{(r,y,z)}(t,x).
\end{equation}
Consequently, if $p\geq 2$ is such that $m_p<\infty$ and \eqref{def-Mp} holds, then %by \eqref{mom-v} and \eqref{2Gv},
\begin{equation}
\label{2Gv-1}
\sup_{r\leq t\leq T}\sup_{x,y\in \bR}\|v^{(r,y,z)}(t,x)\|_p \leq 2C_{T,p,\nu,\k}G_{t-r}(x-y)|z|.
\end{equation}
In particular, under Assumption B, relations \eqref{mom-v} and \eqref{2Gv-1} hold for any $p\geq 2$ such that $m_p<\infty$.
\end{proposition}

\begin{proof}
a) We use Picard iterations: we let $v_0(t,x)=G_{t-r}(x-y)z$, and for $n\geq 0$, 
\begin{equation}
\label{def-vn}
v_{n+1}(t,x)=G_{t-r}(x-y)z+\int_r^t \int_{\bR}G_{t-s}(x-y')v_n(s,y')X(ds,dy').
\end{equation}
Letting $v_{-1}=0$, we see that for any $n\geq 0$,
\begin{equation}
\label{Picard-v}
v_{n+1}(t,x)-v_n(t,x)=\int_r^t \int_{\bR}G_{t-s}(x-y')\big(v_n(s,y')-v_{n-1}(s,y')\big) X(ds,dy').
\end{equation}

We proceed with the argument for $p\geq 2$. The same argument for $p=2$ gives the existence of the solution, and proves the estimate \eqref{mom-v} for $p=2$. By Theorem \ref{mom-SZ-th},
\begin{align*}
& \bE|v_{n+1}(t,x)-v_n(t,x)|^p \\
& \quad \leq  \cC_{p,\nu}(t) \int_r^t \sup_{y' \in \bR}\bE|v_n(s,y')-v_{n-1}(s,y')|^p \Big(J_2(t-s)+ \big(J_p(t-s)\big)^{p/2}\Big) ds,
\end{align*}
where 
\[
J_p(t)=\|G_t*\k\|_{L^p(\bR)}^2\]
and
\[
J_2(t)=\int_{\bR}|(G_t*\k)(x)|^2 dx=\int_{\bR}\int_{\bR} G_t(x)G_t(y) f(x-y)dxdy.
\]
The constant  $\cC_{p,\nu}(t)=\cC_p \max(q_t^{p/2-1},1)$ is non-decreasing in $t$, $\cC_p$ is being the constant from Proposition \ref{ros-prop} (which depends on $\nu$ through $m_2$ and $m_p$), and
\[
q_t:=\int_0^t \int_{\bR}\int_{\bR}G_{t-s}(x-y)G_{t-s}(x-z)f(y-z)dydz =\int_0^t \int_{\bR}\frac{\sin^2(s|\xi|)}{|\xi|^2}\mu(d\xi)ds.
\]

We fix $T>0$. For any $t\in [r,T]$ and $n\geq 0$, let $H_n(t)=\sup_{x\in \bR}\|v_n(t,x)-v_{n-1}(t,x)\|_p$. 
Then, for any $t \in [r,T]$ and $n\geq 0$, we have:
\[
H_{n+1}(t) \leq \cC_{p,\nu}(T)\int_0^t H_n(s)\Big(J_2(t-s)+ \big(J_p(t-s)\big)^{p/2}\Big) ds.
\]
Note that
\[
\int_0^T \Big(J_2(t)+ \big(J_p(t)\big)^{p/2}\Big) dt<\infty,
\]
since $\int_0^T J_2(t)dt=q_T<\infty$ and $\int_0^T \big(J_p(t)\big)^{p/2} dt=\cM_p(T)<\infty$. Moreover, 
\[
M:=\sup_{t \in [r,T]}H_0(t) =\sup_{t \in [r,T]} \sup_{x\in \bR}G_{t-r}^p(x-y)|z|^p=2^{-p}|z|^p
\]

By Lemma 15 of \cite{dalang99} with $k_1=k_2=0$ and  $g(t)=\cC_{p,\nu}(T)\big(J_2(t)+ \big(J_p(t)\big)^{p/2}\big)$, there exists a sequence $(a_n)_{n\geq 1}$ of positive numbers depending on $T$ and $g$ (which in turn depends on $(p,\nu,\k)$), such that $\sum_{n\geq 1}a_n^{1/q}<\infty$ for any $q\geq 1$, and
\[
H_n(t) \leq M a_n \quad \mbox{for any} \quad t \in [r,T],n\geq 1.
\]
More precisely, $a_n=G(T)P(S_n\leq T)$ where $G(T)=\int_0^T g(t)dt$, $S_n=\sum_{i=1}^n X_i$, and $(X_i)_{i\geq 1}$ are i.i.d. random variables on $[0,T]$ with density $g(t)/G(T)$.

This means that
\begin{equation}
\label{series-vn}
\sum_{n\geq 1}\sup_{t \in [r,T]}\sup_{x\in \bR} \|v_n(t,x)-v_{n-1}(t,x)\|_{p} \leq \frac{1}{2} |z|\sum_{n \geq 1} a_n^{1/p}<\infty.
\end{equation}
Hence, $\{v_n(t,x)\}_{n\geq 1}$ is Cauchy in $L^p(\Omega)$, uniformly in $(t,x)\in [r,T] \times \bR$. We denote its limit by $v(t,x)$. Letting $n \to \infty$ in \eqref{Picard-v}, we infer that $v$ satisfies \eqref{HAM-delta-int}. The uniqueness of the solution follows by similar methods.

To prove \eqref{mom-v}, we argue as follows:
\begin{align*}
\|v(t,x)\|_p \leq \|v(t,x)-v_n(t,x)\|_p +\sum_{k=1}^{n}\|v_k(t,x)-v_{k-1}(t,x)\|_p+\|v_0(t,x)\|_p.
\end{align*}

Taking the supremum over 
$(t,x) \in [r,T] \times \bR$ and using \eqref{series-vn}, we get:
\[
\sup_{t \in [r,T]}\sup_{x\in \bR} \|v(t,x)\|_p  \leq \sup_{t \in [r,T]}\sup_{x\in \bR}\|v(t,x)-v_n(t,x)\|_p +\frac{1}{2}|z|\sum_{n\geq 1}a_n^{1/p}+\frac{1}{2}|z|
\]
Letting $n \to \infty$, we obtain estimate \eqref{mom-v} with $C_{T,p,\nu,\k}=\frac{1}{2}(\sum_{n\geq 1}a_n^{1/p}+1)$.

b) From part a), we know that $v(t,x)=\lim_{n \to \infty}v_n(t,x)$ in $L^2(\Omega)$. From  definition \eqref{def-vn} of $v_{n+1}(t,x)$, we have:
\begin{align*}
v_{n+1}(t,x)&=zG_{t-r}(x-y)+\int_r^t \int_{\bR} \big(G_{t-s}(x-\cdot) v_n(s,\cdot) * \k\big)(y') L(ds,dy')\\
&=zG_{t-r}(x-y)+\int_r^t \int_{\bR} \int_{\bR_0}\big(G_{t-s}(x-\cdot) v_n(s,\cdot) * \k\big)(y') z'\widehat{N}(ds,dy',dz').
\end{align*}
Since the integrand is predictable, this It\^o integral coincides with the Skorohod integral, by Lemma 2.5.(iv) of \cite{BZ24}. Hence,
\[
v_{n+1}(t,x)=zG_{t-r}(x-y)+\delta(V_{t,x,n}),
\]
where $V_{t,x,n}(s,y',z')=\big(G_{t-s}(x-\cdot) v_n(s,\cdot) * \k\big)(y') z'$. 
We claim that:
\begin{equation}
\label{claim1}
v_{n}(t,x)=\sum_{k=0}^n I_{k}\big(g_k^*(\cdot,r,y,z,t,x)\big) \quad \mbox{for all} \quad n\geq 0.
\end{equation}
Then, letting $n\to \infty$ in $L^2(\Omega)$, we obtain that $v(t,x)$ has the desired chaos expansion.

To prove \eqref{claim1}, we proceed by induction.
For $n=0$, this is clear since $v_0(t,x)=zG_{t-r}(x-y)=g_0^*(r,y,z,t,x)$. For the induction step, we assume that the statement holds for $n$ and we prove it for $n+1$. Using the induction hypothesis, we have:
\begin{align*}
 V_{t,x,n}(s,y',z')&=z'\int_{\bR}G_{t-s}(x-y'')v_n(s,y'')\k(y'-y'')dy''\\
&=z'z\int_{\bR}G_{t-s}(x-y'')G_{s-r}(y''-y)\k(y'-y'')dy''+\\
& \quad  \sum_{k=1}^n z'\int_{\bR}
G_{t-s}(x-y'')I_{k}\big(g_k^*(\cdot,r,y,z,t,x)\big) \k(y'-y'')dy''\\
&=g_1^*(r,y,z,t,x)+\sum_{k=1}^{n}I_k\big(g_{k+1}^*(\cdot,s,y',z',r,y,z,t,x)\big),
\end{align*}
where the last line follows by stochastic Fubini theorem and definition \eqref{def-gn*} of $g_n^*$. By Lemma \ref{Skor-lem}, it follows that
\[
\delta(V_{t,x,n})=\sum_{k=0}^{n}I_{k+1}\big(g_{k+1}^*(\cdot,s,y',z',r,y,z,t,x)\big)=
\sum_{k=1}^{n+1}I_{k}\big(g_{k}^*(\cdot,s,y',z',r,y,z,t,x)\big).
\] 
This proves that the desired statement holds for $n+1$.

c) Using the result in part b), we have:
\begin{align*}
G_{t-r}(x-y)v^{(r,y,z)}(t,x)=\sum_{n\geq 0}I_n\big(G_{t-r}(x-y)g_n^*(\cdot,r,y,z,t,x)\big).
\end{align*}
The conclusion follows using the fact that
$
G_{t-r}(x-y)g_n^*(\cdot,r,y,z,t,x)=\frac{1}{2}g_n^*(\cdot,r,y,z,t,x)$, because
$g_n(t_1,x_1,z_1,\ldots,t_n,x_n,z_n,r,y,z,t,x)$ is equal to a constant times the indicator of the set $\{|x-x_n|<t-t_n,\ldots,|x_2-x_1|<t_2-t_1,|x_1-y|<t_1-r\}$, which is included in the set $\{|x-y|<t-r\}$.
\end{proof}
%% END MDsol

%% BEGIN MD2sol
\section{Second Malliavin derivative of the solution}
\label{section-Mall2}

In this section, we prove a similar estimate for the $p$-th moment of the second Malliavin derivative.

Recall that the solution $u$ has the Poisson chaos expansion \eqref{chaos-Du}.

\begin{proposition}
For any $t>0$ and $x \in \bR$, $u(t,x) \in {\rm dom}(D^2)$.
\end{proposition}

\begin{proof}
We proceed as in the proof of Proposition \ref{u-Mall1}.
From Proposition 7.1 of \cite{sanz05}, we know that
$U(t,x)$ is twice Malliavin differentiable with respect to $W$. Based on the Wiener-chaos expansion \eqref{chaos-UW}, we deduce that the second Malliavin derivative of $U(t,x)$ (with respect to $W$) is given by:
\[
D_{(r_1,y_1),(r_2,y_2)}^{2,W} U(t,x)=1+\sum_{n\geq 2}m_2^{n/2}n(n-1)I_{n-2}^W\big(\widetilde{f_n}(\cdot,r_1,y_1,r_2,y_2,t,x)\big), \quad r_1,r_2 \in [0,t],y_1,y_2 \in \bR.
\]
Using relation \eqref{Poi-Gauss}, we infer that
\[
\sum_{n\geq 2}n(n-1) n! \|\widetilde{f_n^*}(\cdot,t,x)\|_{\cH^{\otimes n}}^2 =\sum_{n\geq 2}n(n-1) n! \, m_2^{n} \|\widetilde{f_n}(\cdot,t,x)\|_{\cH_0^{\otimes n}}^2 =\bE\|D^{2,W} U(t,x)\|_{\cH_0^{\otimes 2}}^2<\infty.
\]
The fact that $u(t,x)\in {\rm dom}(D^2)$ now follows; see \eqref{check-D2}.

\end{proof}

Let $0<r_1<r_2<t$, $y_1,y_1 \in \bR$ and $z_1,z_2 \in \bR_0$. Then 
\begin{equation}
\label{chaos-D2u}
D_{(r_1,y_1,z_1),(r_2,y_2,z_2)}^2 u(t,x)=\sum_{n\geq 2}n(n-1)I_{n-2}(\widetilde{f}_n^{*}(\cdot,(r_1,y_1,z_1,r_2,y_2,z_2,t,x)).
\end{equation}

The first term in Poisson chaos expansion above is 
\[
2\widetilde{f}_2^{*}(r_1,y_1,z_1,r_2,y_2,z_2,t,x)=2z_1 z_2 \int_{\bR^2} f_2(r_1,y_1',r_2,y_2',t,x)\k(y_1-y_1')\k(y_2-y_2')dy_1'dy_2'.
\]
We will show that this term dominates the other terms, when estimating the $p$-th moment.

\medskip

As in the case of the Malliavin derivative of $u(t,x)$, we start with the representation of the kernel $f_{n}^*(\cdot,t,x)$ with fixed arguments $(r_1,y_1,z_1)$ and $(r_2,y_2,z_2)$ in the last two positions: we write
\begin{equation}
\label{fn*2}
\widetilde{f}_n^{*}(\cdot,r_1,y_1,z_1,r_2,y_2,z_2,t,x)=\frac{1}{n(n-1)}\sum_{i,j=1,i<j}^{n}
h_{ij}^{*(n)}(\cdot,r_1,y_1,z_1,r_2,y_2,z_2,t,x),
\end{equation}
with $h_{ij}^{*(n)}(\cdot,r_1,y_1,z_1,r_2,y_2,z_2,t,x)$ the symmetrization of the function $f_{ij}^{*(n)}(\cdot,r_1,y_1,z_1,r_2,y_2,z_2,t,x)$ obtained by placing $(r_1,y_1,z_1)$ and $(r_2,y_2,z_2)$ on the $i$-th and $j$-th positions, respectively, among the arguments of $f_n^*(\cdot,t,x)$. More precisely, if $\xi_{\ell}=(t_{\ell},x_{\ell},\zeta_{\ell})$ for $\ell \leq n-2$, then
\begin{align*}
& f_{ij}^{*(n)}(\xi_1,\ldots,\xi_{n-2},r_1,y_1,z_1,r_2,y_2,z_2,t,x)\\
&  \quad :=f_n^{*}(\xi_1,\ldots,\xi_{i-1},r_1,y_1,z_1,\xi_{i},\ldots,\xi_{j-2},r_2,y_2,z_2,\xi_{j-1},\ldots,
\xi_{n-2,}t,x)\\
& \quad = z_1 z_2 \prod_{\ell=1}^{n-2}\zeta_{\ell} \int_{\bR^{2}} \k(y_1-y_1') \k(y_2-y_2')\\
& \quad \quad \left(\int_{\bR^{n-j}} G_{t-t_{n-2}}(x-w_{n-2}) \ldots G_{t_{j-1}-r}(w_{j-1}-y_2')
\prod_{\ell=j-1}^{n-2}\k(x_{\ell}-w_{\ell}) dw_{j-1}\ldots dw_{n-2}\right) \\
& \quad \quad \left(\int_{\bR^{n-j}} G_{r-t_{j-2}}(y_2'-w_{j-2}) \ldots G_{t_{i}-\theta}(w_{i}-y_1') 
\prod_{\ell=i}^{j-2}\k(x_{\ell}-w_{\ell}) dw_{i}\ldots dw_{j-2}\right) \\
& \quad \quad \left(\int_{\bR^{i-1}} G_{\theta-t_{i-1}}(y_1'-w_{i-1}) \ldots G_{t_{2}-t_1}(w_2-w_1) 
\prod_{\ell=1}^{i-1}\k(x_{\ell}-w_{\ell}) dw_{1}\ldots dw_{i-1}\right) dy_1'dy_2'.
\end{align*}

Recalling notation \eqref{def-gn*}, we obtain the following representation: for any $1\leq i<j \leq n$,
\begin{align*}
& f_{ij}^{*(n)}(\xi_1,\ldots,\xi_{n-2},r_1,y_1,z_1,r_2,y_2,z_2,t,x)=z_1 z_2 \int_{\bR^{2}} \k(y_1-y_1') \k(y_2-y_2') f_{i-1}^*(\xi_1,\ldots,\xi_{i-1},r_1,y_1')\\
& \qquad \qquad \qquad \qquad g_{j-i-1}^*(\xi_i,\ldots,\xi_{j-2},r_1,y_1',z_1,r_2,y_2',z_2) g_{n-j}^*(\xi_{j-1},\ldots,\xi_{n-2},r_2,y_2',z_2,t,x)dy_1'dy_2'.
\end{align*}

\begin{theorem}
\label{key-thD2}
Let $p\geq 2$ be such that $m_p<\infty$. If \eqref{def-Mp} holds, then for any $T>0$, $t,r_1,r_2 \in [0,T]$, $x,y_1,y_2 \in \bR$ and $z_1,z_2 \in \bR_0$,
\begin{equation}
\label{key-D2}
\|D_{(r_1,y_1,z_1),(r_2,y_2,z_2)}^2 u(t,x) \|_p \leq C_{T,p,\nu,\k}'' |z_1 z_2| \int_{\bR^2} \widetilde{f}_2(r_1,y_1',r_2,y_2',t,x) \k(y_1-y_1')\k(y_2-y_2') dy_1'dy_2',
\end{equation}
where $C_{T,p,\nu,\k}''= 8K_p(T) (C_{T,p,\nu,\k})^2$ with $K_p(T)$ and $C_{T,p,\nu,\k}$ given by \eqref{mom-u}, respectively \eqref{mom-v}. In particular, under Assumption B, relation \eqref{key-D2} holds for any $p\geq 2$ such that $m_p<\infty$.
\end{theorem}

\begin{proof} We proceed as in the proof of Theorem \ref{key-thD}. Let $0<r_1 < r_2<t$. Using \eqref{chaos-D2u} and \eqref{fn*2}, we have:
\begin{align*}
D^2_{(r_1,y_1,z_1),(r_2,y_2,z_2)}u(t,x) & = \sum_{n\geq 1} \sum_{i,j=1, i<j}^{n}I_{n-2}\big(f_{ij}^{*(n)}(\cdot,r_1,y_1,z_1,r_2,y_2,z_2,t,x)\big)
\\
& = \sum_{i \geq 1} \sum_{j > i} \sum_{n\geq j}I_{n-2}\big(f_{ij}^{*(n)}(\cdot,r_1,y_1,z_1,r_2,y_2,z_2,t,x)\big).
\end{align*}

We use the stochastic Fubini theorem given by Lemma 2.6 of \cite{BZ24}, and we apply (twice) the product formula given by Lemma \ref{prod-lem}. We obtain:
\begin{align*}
& I_{n-2}\big(f_{ij}^{*(n)}(\cdot,r_1,y_1,z_1,r_2,y_2,z_2,t,x)\big) 
\\
& = \int_{\bR^2} dy_1' dy_2' \, \k(y_1-y_1') \k(y_2 - y_2') 
\\
& \qquad I_{n-2}\big(f_{i-1}^*(\cdot,r_1,y_1')\otimes g_{j-i-1}^*(\cdot,r_1,y_1',z_1,r_2,y_2') \otimes g_{n-j}^*(\cdot,r_2,y_2',z_2,t,x) \big) 
\\
& = \int_{\bR^2}dy_1'dy_2' \, \k(y_1-y_1') \k(y_2-y_2') 
\\
& \qquad I_{i-1}\big(f_{i-1}^*(\cdot,r_1,y_1')\big)  I_{j-i-1}\big( g_{j-i-1}^*(\cdot,r_1,y_1',z_1,r_2,y_2')\big)I_{n-j}\big( g_{n-j}^*(\cdot,r_2,y_2',z_2,t,x)\big). 
\end{align*}

Therefore,
\begin{align*}
& D^2_{(r_1,y_1,z_1),(r_2,y_2,z_2)}u(t,x) 
\\
&= \int_{\bR^2}  dy_1'dy_2' \, \k(y_1-y_1')k(y_2-y_2') \left(\sum_{i\geq 1} I_{i-1}\big(f_{i-1}^*(\cdot,r_1,y_1') \big) \right) 
\\
& \qquad  \left(\sum_{j > i} I_{j-i-1}\big( g_{j-i-1}^*(\cdot,r_1,y_1',z_1,r_2,y_2')\big) \right) \left(\sum_{n \geq j} I_{n-j}\big( g_{n-j}^*(\cdot,r_2,y_2',z_2,t,x)\big) \right).
\end{align*}

The first series is the chaos expansion \eqref{chaos-u} of $u(r_1,y_1')$. The second and third series are the chaos expansions of $v^{(r_1,y_1',z_1)}(r_2,y_2')$ and $v^{(r_2,y_2',z_2)}(t,x)$ given by $\eqref{HAM-delta}$ (see \eqref{chaos-v}). It follows that: 
\[
D^2_{(r_1,y_1,z_1),(r_2,y_2,z_2)}u(t,x)=\int_{\bR^2} u(r_1,y_1')v^{(r_1,y_1',z_1)}(r_2,y_2')v^{(r_2,y_2',z_2)}(t,x)\k(y_1-y_1') \k(y_2-y_2')dy_1' dy_2'.
\]
Note that the variables $u(r_1,y_1')$ and $v^{(r_1,y_1',z_1)}(r_2,y_2')$ and $v^{(r_2,y_2',z_2)}(t,x)$ are independent since they are measurable with respect to $\cF_{r_1}$, $\cF_{r_1,r_2}$, respectively $\cF_{r_2,t}$, and these $\sigma$-fields are independent since the sets $[0,r_1]$, $(r_1,r_2]$ and $(r_2,t]$ are disjoint. Here $\cF_{r,t}$ is given by \eqref{def-Frt}.
Therefore, using Minkowski inequality, followed by \eqref{mom-u} and \eqref{2Gv-1}, we have:
\begin{align*}
& \|D^2_{(r_1,y_1,z_1),(r_2,y_2,z_2)}u(t,x)\|_p 
\\
& \leq \int_{\bR^2} \|u(r_1,y_1')\|_p\|v^{(r_1,y_1',z_1)}(r_2,y_2')\|_p\|v^{(r_2,y_2',z_2)}(t,x)\|_p\k(y_1-y_1') \k(y_2-y_2') dy_1' dy_2'
\\
& \leq 4 K_p(T) (C_{T,p,\nu,\k})^2|z_1z_2|\int_{\bR^2}G_{t-r_2}(x-y_2')G_{r_2-r_1}(y_2'-y_1') \k(y_1-y_1') \k(y_2-y_2') dy_1' dy_2.
\end{align*}
The same argument shows for $r_2 < r_1$ that, 
\begin{align*}
 & \|D^2_{(r_1,y_1,z_1),(r_2,y_2,z_2)}u(t,x)\|_p  
 \\
 & \leq 4 K_p(T) (C_{T,p,\nu,\k})^2|z_1z_2|\int_{\bR^2}G_{t-r_1}(x-y_1')G_{r_1-r_2}(y_1'-y_2') \k(y_1-y_1') \k(y_2-y_2') dy_1' dy_2.
\end{align*}
The cases $r_1 < r_2$ and $r_2 < r_1$ combine to give \eqref{key-D2} (see \eqref{fn-defn}). 
\end{proof}
%% END MD2sol

%% BEGIN Ergodicity
\section{Proofs of the main results}
\label{section-proofs}

In this section, we include the proofs of Theorems \ref{ergodic-th}, \ref{cov-th}, \ref{QCLT} and \ref{FCLT}.

\medskip

First, we introduce the function $\varphi_{t,R}$ and mention some of its properties which will be used below.
For any $0\leq r \leq t$, $y \in \bR$ and $R>0$, we denote
\begin{equation}
\label{def-phi-tR}
\varphi_{t,R}(r,y)=\int_{-R}^R G_{t-r}(x-y)dx \quad \mbox{and} \quad \varphi_{t,R}(y)=\varphi_{t,R}(0,y).
\end{equation}

\begin{lemma}
a) For any $0 \leq r\leq t$ and $R>0$, we have:
\begin{equation}
\label{norm-phi}
\|\varphi_{t,R}(r, \cdot)\|_{L^p(\bR)} \les R^{1/p}(t-r) \quad \mbox{for any $p\geq 1$}.
\end{equation} 
In particular for $r = 0$,
\[
\| \varphi_{t,R} \|_{L^p(\bR)} \les R^{1/p}.
\]
b) For any $0  \leq r\leq s\leq t$ and $R>0$, we have:
\begin{equation}
\label{tsPhiDif}
\|\varphi_{t,R}(r, \cdot)-\varphi_{s,R}(r, \cdot)\|_{L^p(\bR)} \les R^{1/p}(t-s) \quad \mbox{for any $p\geq 1$}.
\end{equation} 
\end{lemma}

\begin{proof}
a) This follows by H\"older inequality and the fact that $\int_{\bR}G_t(x)dx=t$:
\begin{align*}
\|\varphi_{t,R}(r,\cdot)\|_{L^p(\bR)}^{p}&=
\int_{\bR}\left(\int_{-R}^R G_{t-r}(x-y)dx\right)^p dy \leq (t-r)^{p-1} \int_{\bR}\int_{-R}^{R}G_{t-r}^p(x-y)dxdy\\
&=\frac{(t-r)^{p-1}}{2^{p-1}} \int_{-R}^{R}
\left(\int_{\bR}G_{t-r}(x-y)dy\right)dx = 2^{2-p} (t-r)^p R.
\end{align*}

b) We write
\[
\varphi_{t,R}(r, \cdot)-\varphi_{s,R}(r, \cdot)=\frac{1}{2}\int_{-R}^R 1_{\{s-r<|x-y|<t-r\}} dx=\int_{\bR}1_{\{|x|<R\}} \mu(dx),
\]
where $\mu(dx)=\frac{1}{2}1_{\{s-r<|x-y|<t-r\}}$, with $\mu(\bR)=t-s$. By H\"older's inequality,
\[
\big|\varphi_{t,R}(r, \cdot)-\varphi_{s,R}(r, \cdot)\big|^p \leq (t-s)^{p-1} \int_{-R}^R 1_{\{s-r<|x-y|<t-r\}} dx.
\]
Hence, by Fubini's theorem,
\[
\int_{\bR}\big|\varphi_{t,R}(r, \cdot)-\varphi_{s,R}(r, \cdot)\big|^p dx \leq (t-s)^{p-1} \int_{-R}^R \left(\int_{\bR}1_{\{s-r<|x-y|<t-r\}} dy\right) dx=4R (t-s)^p.
\]
\end{proof}

If $\k \in L^1(\bR)$, then by Young's inequality and \eqref{norm-phi}, for any $p\geq 1$,

\begin{equation}
\label{young-phi}
\| \varphi_{t,R}*\k \|_{L^p(\bR)} \leq \|\varphi_{t,R}\|_{L^p(\bR)} \|\k\|_{L^1(\bR)} \les R^{1/p}.
\end{equation}

If $\k =R_{1,\alpha/2}$ for some $\alpha \in (0,1)$, then by Theorem \ref{HLS} and \eqref{norm-phi}, for any $p>\frac{2}{2-\alpha}$
\begin{equation}
\label{HLS-phi}
\| \varphi_{t,R}*\k \|_{L^p(\bR)} \les \|\varphi_{t,R}\|_{L^q(\bR)}  \les R^{1/q},
\end{equation}
where $\frac{1}{q}=\frac{1}{p}+\frac{\alpha}{2}$.

\subsection{Ergodicity}

In this section, we give the proof of Theorem \ref{ergodic-th}. This
follows from the key estimate for the Malliavin derivative of $u$ given by Theorem \ref{key-thD}, and the lemma below.

\begin{lemma}
\label{lem-key-Z}
Suppose that Assumption B holds.
Let $\{Z(t,x);t\geq 0,x\in \bR\}$ be an adapted random field such that
$\{Z(t,x)\}_{x\in \bR^d}$ is strictly stationary and $Z(t,x) \in {\rm dom}(D)$ for any $t\geq 0$ and $x \in \bR$. Assume that for any $0<r<t$, $x,y \in \bR$ and $z \in \bR_0$,
\begin{equation}
\label{key-Z}
\|D_{r,y,z}Z(t,x)\|_{2} \leq C_t |z| \int_{\bR} G_{t-r}(x-y')\k(y-y')dy',
\end{equation}
where $C_t>0$ is a constant depending on $t$. Then $\{Z(t,x)\}_{x\in \bR}$ is ergodic.
\end{lemma}

\begin{proof} By Lemma 4.2 of \cite{BZ25}, it suffices to show that for any $k\geq 1$, and $b_1, \ldots , b_k, \zeta_1, \ldots, \zeta_k\in\bR$, the following relation holds:
\begin{align}
\label{ERG1}
\lim_{R\to\infty} \frac{1}{R^{2}}
{\rm Var}\bigg(  \int_0^R  \cos \bigg(\sum_{j=1}^k b_j  Z(t,x + \zeta_j)  \bigg) dx \bigg)
=0,
\end{align}
together with a similar relation in which $\cos$ is replaced by $\sin$. We denote by $V_R$ the variance appearing in \eqref{ERG1}. By Poincar\'e inequality \eqref{poincare}, 
\begin{align*}
V_R & \leq E \left\| D \left( \int_0^R  \cos \bigg(\sum_{j=1}^k b_j  Z(t,x + \zeta_j)  \bigg) dx \right) \right\|_{\cH}^2 \\
&= \int_0^t \int_{\bR} \int_{\bR_0} \left\|\int_0^R D_{r,y,z} \cos \bigg(\sum_{j=1}^k b_j  Z(t,x + \zeta_j) dx  \bigg) \right\|_2^2 drdy \nu(dz)\\
& \leq \int_0^t \int_{\bR} \int_{\bR_0} \left( \int_0^R \bigg\|D_{r,y,z} \cos \bigg(\sum_{j=1}^k b_j  Z(t,x + \zeta_j)\bigg) \bigg\|_2 dx  \right)^2 drdy \nu(dz).
\end{align*}
By \eqref{chain1} and \eqref{key-Z},
\begin{align*}
& \left(\int_0^R \bigg\|D_{r,y,z} \cos \bigg(\sum_{j=1}^k b_j  Z(t,x + \zeta_j)\bigg) \bigg\|_2 dx \right)^2 \leq
\left( \int_0^R  \bigg\|\sum_{j=1}^k b_j  D_{r,y,z} Z(t,x + \zeta_j) \bigg\|_2 dx \right)^2 \\
 & \quad \leq \left(\sum_{j=1}^k |b_j| \int_0^R \big\| D_{r,y,z} Z(t,x + \zeta_j) \big\|_2 dx\right)^2 \leq k \sum_{j=1}^k b_j^2 \left( \int_0^R \big\| D_{r,y,z} Z(t,x + \zeta_j) \big\|_2 dx\right)^2 \\
 & \quad \leq  k C_t^2 |z|^2 \sum_{j=1}^k b_j^2 \left( \int_0^R  \int_{\bR} G_{t-r}(x+\zeta_j-y') \k(y-y') dy' dx\right)^2 \\
 & \quad \leq  k C_t^2 |z|^2 \sum_{j=1}^k b_j^2 \left(  \int_{\bR} \varphi_{t,R}(y'-\zeta_j) \k(y-y') dy' \right)^2  = k C_t^2 |z|^2 \sum_{j=1}^k b_j^2 \big[\big(\varphi_{t,R}*\k)(y-\zeta_j\big)\big]^2
\end{align*}
where $\varphi_{t,R}$ is given by \eqref{def-phi-tR}. Hence, 
\[
V_R \leq  k C_t^2 m_2 t \sum_{j=1}^k b_j^2 \int_{\bR}\big[\big(\varphi_{t,R}*\k)(y-\zeta_j\big)\big]^2 dy=k C_t^2 m_2 t \sum_{j=1}^k b_j^2 \big\|\varphi_{t,R}*\k\big\|_{L^2(\bR)}^2 \les R^{\beta},
\]
where for the last inequality we used \eqref{young-phi} if  $k \in L^1(\bR)$, respectively \eqref{HLS-phi} if $\k=R_{1,\alpha/2}$, and $\beta$ is defined by \eqref{def-beta}. This concludes the proof of \eqref{ERG1}.

\end{proof}
%% END Ergodicity

%% BEGIN Limcov
\subsection{Limiting covariance}

In this section, we give the proof of Theorem \ref{cov-th}. 

\medskip

Similarly to \eqref{Poi-Gauss}, it can be proved that for any $t,s\in \bR_{+}$, $x,y \in \bR$ and $n\geq 1$,
\[ \langle \widetilde{f}_n^*(\cdot,t,x),\widetilde{f}_n^*(\cdot,s,y)\rangle_{\cH^{\otimes n}}= m_2^n \, \langle \widetilde{f}_n(\cdot,t,x),\widetilde{f}_n(\cdot,s,y)\rangle_{\cH_0^{\otimes n}}.
\]

Therefore,
\begin{align*}
& \bE\big[\big(u(t,x)-1\big)\big(u(s,y)-1\big)\big]=\sum_{n\geq 1}n!\,\langle \widetilde{f}_n^*(\cdot,t,x),\widetilde{f}_n^*(\cdot,s,y)\rangle_{\cH^{\otimes n}}\\
& \quad = \sum_{n\geq 1}n!\, m_2^n \, \langle \widetilde{f}_n(\cdot,t,x),\widetilde{f}_n(\cdot,s,y)\rangle_{\cH_0^{\otimes n}}=\bE\big[\big(U(t,x)-1\big)\big(U(s,y)-1\big)\big],
\end{align*}
where $U=\{U(t,x);t\geq 0,x\in \bR\}$ is the solution of \eqref{HAM-W}.
It follows that
\begin{align*}
\bE\big[ F_R(t) F_R(s)\big] & =\int_{-R}^R \int_{-R}^R \bE\big[\big(u(t,x)-1\big)\big(u(s,y)-1\big)\big] dxdy =
% = \int_{-R}^R \int_{-R}^R \bE\big[\big(U(t,x)-1\big)\big(U(s,y)-1\big)\big]=
\bE\big[ F_R'(t) F_R'(s)\big],
\end{align*}
where 
\[
F_R'(t)=\int_{-R}^R \big(U(t,x)-1\big)dx.
\]

If $k\in L^1(\bR)$, then by relation (4.1) of \cite{NZ22}, 
\[
\lim_{R\to \infty} \frac{1}{R}\bE\big[ F_R'(t) F_R'(s)\big] =2\int_{\bR} {\rm Cov}\big(U(t,x),U(s,0)\big)dx=:K(t,s) \quad \mbox{is finite}.
\]

If $\k=R_{1,\alpha/2}$ for some $\alpha \in (0,1)$, then by Remark 2 of \cite{DNZ20},
%with $H=\frac{\alpha+1}{2}$,
\[
\lim_{R \to \infty}\frac{1}{R^{\alpha+1}}\bE\big[ F_R'(t) F_R'(s)\big]=c_{\alpha} \int_0^{t\wedge s}
(t-r)(s-r)dr=:K(t,s)
\]
where $c_{\alpha}>0$ is a constant depending on $\alpha$. This concludes the proof.
%% END Limcov

%% BEGIN QCLT
\subsection{Quantitative CLT}

In this section, we give the proof of Theorem \ref{QCLT}. We will use the following result.

\begin{proposition}[Theorem 3.4 of \cite{trauthwein25}]
\label{tara}
Let $F\in {\rm dom}(D)$ with $\bE[F]=0$ and  ${\rm Var}(F)=\sigma^2 > 0$. Then, 
for any $p,q\in (1,2]$,
\begin{align*}
d_{\rm FM}\left( \frac{F}{\sigma}, Z\right) & \leq d_{\rm W}\left(\frac{F}{\sigma}, Z\right)  \leq    \gamma_1 +  \gamma_2 +  \gamma_3\\
  d_{\rm K}\left(\frac{F}{\sigma}, Z\right) & \leq  \sqrt{ \frac{\pi}{2}} (  \gamma_1 +  \gamma_2)   +  \gamma_4 +  \gamma_5+  \gamma_6 +  \gamma_7,
\end{align*}
where  $Z\sim N(0,1)$ and the seven quantities $ \gamma_1, \ldots,  \gamma_7$ are given as follows:
\begin{align}
\begin{aligned}
 \gamma_1&:=  \frac{2^{\frac2p + \frac12} }{\sqrt{\pi}}
 \sigma^{-2} \bigg( \int_{Z}  \bigg[ \int_Z  \| D_{\xi_2}F \|_{2p}  \| D^2_{\xi_1,\xi_2}F\|_{2p} 
  \,  \fm(d\xi_2) \bigg]^p
  \fm(d\xi_1) \bigg)^{\frac1p}
\\
 \gamma_2&:=\frac{2^{\frac2p - \frac12} }{\sqrt{\pi}}
\sigma^{-2} \bigg( \int_{Z}  \bigg[ \int_Z   \| D^2_{\xi_1,\xi_2}F\|^2_{2p}  \,  \fm(d\xi_2) \bigg]^p
  \fm(d\xi_1) \bigg)^{\frac1p}
\\
 \gamma_3&:= 2 \sigma^{-(q+1)}\int_Z \| D_\xi F \|_{q+1}^{q+1}  \, \fm(d\xi)
\\
 \gamma_4&:= 2^{\frac{2}{p}} \sigma^{-2} \bigg( \int_Z \| D_\xi F \|_{2p}^{2p}  \, \fm(d\xi) \bigg)^{\frac1p}
\\
 \gamma_5&:= (4p)^{\frac1p} \sigma^{-2} \bigg( \int_{Z^2} \| D^2_{\xi_1,\xi_2}  F \|_{2p}^{2p}
\, \fm(d\xi_1)\fm(d\xi_2)  \bigg)^{\frac1p}
\\
 \gamma_6&:=(2^{2+p}p)^{\frac1p} \sigma^{-2} \bigg( \int_{Z^2} \| D^2_{\xi_1,\xi_2}  F \|_{2p}^{p}
 \| D_{\xi_1}F \|_{2p}^p
\, \fm(d\xi_1)\fm(d\xi_2)  \bigg)^{\frac1p}\\
 \gamma_7&:=  \frac{  (8p)^{\frac1p} }{\sigma^2}
 \bigg( \int_{Z^2} \| D^2_{\xi_1,\xi_2}  F \|_{2p}
 \| D_{\xi_1}F \|_{2p}  \| D_{\xi_2}F \|_{2p}^{2(p-1)}
\, \fm(d\xi_1) \fm(d\xi_2)  \bigg)^{\frac1p}.
\end{aligned}
\label{gamma17}
\end{align}
\end{proposition}

{\em Proof of Theorem \ref{QCLT}:} We apply Proposition \ref{tara} to $F=F_{R}(t)$. 

We recall that $p \in (1,2]$ is such that

\[
m_p<\infty \quad \mbox{and} \quad m_{2p}<\infty.
\]

By Minkowski's inequality, Theorem \ref{key-thD} and the fact that $G_{t-r}(x)\leq G_t(x)$, for any $p'\geq 2$ such that $m_{p'}<\infty$, and for any $r \in [0,t]$, $y \in \bR$ and $z\in \bR_0$
\begin{equation}
\label{Du}
\|D_{r,y,z}F_R(t)\|_{p'} \leq \int_{-R}^R \|D_{r,y,z}u(t,x)\|_{p'}dx \les |z|
\int_{-R}^{R}\int_{\bR}G_{t}(x-y')\k(y-y')dy'dx. %=C |z| \big(\varphi_{t,R} * \k\big)(y).
\end{equation}
Below, we will use this estimate for $p'=2p$ and $p'=q+1$ where $q$ is given by \eqref{def-q}.

\medskip

Similarly, by Minkowski's inequality and Theorem \ref{key-thD2}, for any $p'\geq 2$ such that $m_{p'}<\infty$, and for any $r_1,r_2 \in [0,t]$, $y_1,y_2 \in \bR$ and $z_1,z_2 \in \bR_0$,
\begin{align}
\nonumber
& \|D_{(r_1,y_1,z_1),(r_2,y_2,z_2)}^2 F_R(t)\|_{p'} \leq \int_{R}^{R}\|D_{(r_1,y_1,z_1),(r_2,y_2,z_2)}^2u(t,x)\|_{p'}dx \\
\label{D2u}
& \quad \les |z_1 z_2|
\int_{-R}^R \int_{\bR^2}G_{2t}(x-y_1')G_{2t}(x-y_2')\k(y_1-y_1')\k(y_2-y_2')dy_1'dy_2'dx,
\end{align}
where for the last inequality, we used the fact that
\begin{align*}
\widetilde{f}_2(r_1,y_1,r_2,y_2,t,x) \leq G_{2t}(x-y_1)G_{2t}(x-y_2).
\end{align*}
To see this, we note that if $r_1<r_2<t$
\begin{align*}
& G_{t-r_2}(x-y_2)G_{r_2-r_1}(y_2-y_1)  \leq G_{t}(x-y_2)G_{t}(y_2-y_1)=\frac{1}{4}1_{\{|x-y_2|<t\}} 1_{\{|y_2-y_1|<t\}}\\
& \quad \leq \frac{1}{4} 1_{\{|x-y_2|<t\}} 1_{\{|x-y_1|<2t\}} \leq \frac{1}{4} 1_{\{|x-y_2|<2t\}} 1_{\{|x-y_1|<2t\}}=G_{2t}(x-y_1)G_{2t}(x-y_2).
\end{align*}

Moreover, since
\begin{align*}
\widetilde{f}_2(r_1,y_1,r_2,y_2,t,x) \leq \frac{1}{2} \big[G_{t}(x-y_1)G_{t}(y_1-y_2)+
G_{t}(x-y_2)G_{t}(y_2-y_1)\big]
\end{align*}
for any $p'\geq 2$ such that $m_{p'}<\infty$, and for any $r_1,r_2 \in [0,t]$, $y_1,y_2 \in \bR$ and $z_1,z_2 \in \bR_0$,
\begin{align}
\label{D2u-1}
& \|D_{(r_1,y_1,z_1),(r_2,y_2,z_2)}^2 F_R(t)\|_{p'}  \les |z_1 z_2|
\int_{-R}^R \int_{\bR^2}G_{t}(x-y_1')G_{t}(y_1'-y_2')\k(y_1-y_1')\k(y_2-y_2')dy_1'dy_2'dx.
\end{align}

\medskip

{\bf Case 1.} Assume that \fbox{$\k \in L^1(\bR)$}. In this case, by Theorem \ref{cov-th}, $\sigma_R^2(t) \sim  R$. 

\medskip

We estimate separately $\gamma_1,\ldots,\gamma_7$. We denote $\xi_i=(r_i,y_i,z_i)$ for $i=1,2$, and we use the fact that $D_{r,y,z}F_R(t)=0$ if $r>t$.
%since $F_R(t)$ is $\cF_t$-measurable. 
We use the estimates given by \eqref{Du} and \eqref{D2u}, which are valid for all $r,r_1,r_2 \in [0,t]$. 

We begin by developing some estimates which will be useful for bounding all the $\gamma_i$'s.
% We recall that $m_p$ is defined by \eqref{def-mp}.

%We will use several times H\"older's inequality: for any finite measure space %$(X,\mathcal{X},\mu)$, for any measurable function $f:X \to \bR_{+}$ and $p\geq 1$, 
%\begin{equation}
%\label{Holder}
%\left(\int_{X}f(x)\mu(dx)\right)^p \leq [\mu(X)]^{p-1}\int_X f^p(x)\mu(dx).
%\end{equation}

%Applying inequality \eqref{Holder} to the measure $\mu(dx)=1_{\{|x-y|<t\}}dx$ on $\bR$, we %obtain:

Fix $(r,y,z) \in Z$. By applying \eqref{Du} and integrating with respect to $x$ and then $y'$ we obtain the following rough bound: for any $p'\geq 2$ such that $m_{p'}<\infty$,
\begin{equation}
\label{Du trivial}
\|D_{r,y,z}F_R(t)\|_{p'} \les |z|t\|\k\|_{L^1(\bR)} \les |z|.
\end{equation}

Fix $(r_1,y_1,z_1),(r_2,y_2,z_2) \in Z$. Likewise, by applying \eqref{D2u}, using the fact that $G_{2t}(x-y_2') \leq 1$, and integrating with respect to $x$, $y_1'$ and $y_2'$ (in this order), we obtain: for any $p'\geq 2$ such that $m_{p'}<\infty$,
\begin{align}
\nonumber
& \|D_{(r_1,y_1,z_1),(r_2,y_2,z_2)}^2 F_R(t)\|_{p'} \\
\nonumber
& \les |z_1 z_2| \int_{\bR^2} \int_{-R}^R G_{2t}(x-y_1')\k(y_1-y_1')\k(y_2-y_2')dxdy_1'dy_2' \\
\label{D2u trivial}
& \les |z_1z_2|2t \|\k\|_{L^1(\bR)}^2 \les |z_1z_2|.
\end{align}
Fix $(r_1,y_1,z_1) \in Z$. Let $i \in \bR$, $j \geq 1$ and suppose $m_{i+j} < \infty$. Using \eqref{D2u trivial} and \eqref{D2u}, we obtain: for any $p'\geq 2$ such that $m_{p'}<\infty$,
\begin{align}
\nonumber
& \int_{\bR_0} \int_{\bR} \int_0^t  |z_2|^i\| D^2_{(r_1,y_1,z_1),(r_2,y_2,z_2)} F_R(t) \|_{p'}^j \, dr_2 dy_2 \nu(dz_2)\\
\nonumber
& = \int_{\bR_0} \int_{\bR} \int_0^t  |z_2|^i\| D^2_{(r_1,y_1,z_1),(r_2,y_2,z_2)} F_R(t) \|_{p'}\| D^2_{(r_1,y_1,z_1),(r_2,y_2,z_2)} F_R(t) \|_{p'}^{j-1} \, dr_2 dy_2 \nu(dz_2)\\
\nonumber
& \les\int_{\bR_0} \int_{\bR} \int_0^t   |z_1|^{j-1}  |z_2|^{i+ j-1}\| D^2_{(r_1,y_1,z_1),(r_2,y_2,z_2)} F_R(t) \|_{p'} \, dr_2 dy_2 \nu(dz_2)\\
\nonumber
& \les t|z_1|^j \int_{\bR_0} \int_{\bR}  |z_2|^{i+j} \int_{\bR^2}\int_{-R}^R G_{2t}(x-y_1')G_{2t}(x-y_2')\k(y_1-y_1')\k(y_2-y_2') \, dx dy_1'dy_2'  dy_2 \nu(dz_2).
\end{align}
Integrating with respect to $y_2$ and then $y_2'$, gives: % $z_2$ and then $r_2$, 
\begin{align}
\nonumber
& \int_{\bR_0} \int_{\bR} \int_0^t  |z_2|^i\| D^2_{(r_1,y_1,z_1),(r_2,y_2,z_2)} F_R(t) \|_{p'}^j \, dr_2 dy_2 \nu(dz_2)\\
\nonumber
& \les  |z_1|^j m_{i+j}2t^2 \| \k\|_{L^1(\bR)}   \int_{\bR} \int_{-R}^R G_{2t}(x-y_1')\k(y_1-y_1')  \, dxdy_1'  \\
\label{D2u v2}
& \les  |z_1|^j m_{i+j} (\varphi_{2t,R} * \k)(y_1).
\end{align}

Let $i \in \bR$, $j \geq 1$ and suppose $m_{i+j} < \infty$. Using \eqref{Du trivial}, \eqref{Du} and \eqref{young-phi}, we obtain: for any $p'\geq 2$ such that $m_{p'}<\infty$,
\begin{align}
\nonumber
& \int_{\bR_0} \int_{\bR} \int_0^t |z|^i\| D_{(r,y,z)} F_R(t) \|_{p'}^j \, dr dy \nu(dz) \\
\nonumber
& = \int_{\bR_0} \int_{\bR} \int_0^t |z|^i\| D_{(r,y,z)} F_R(t) \|_{p'}\| D_{(r,y,z)} F_R(t) \|_{p'}^{j-1} \, dr dy \nu(dz) \\
\nonumber
& \les \int_{\bR_0} \int_{\bR} \int_0^t |z|^{i+j-1}\| D_{(r,y,z)} F_R(t) \|_{p'} \, dr dy \nu(dz) \\
\nonumber
& \les \int_{\bR_0} \int_{\bR} \int_0^t |z|^{i+j} \int_{-R}^R \int_{\bR}G_{t}(x-y') \k(y-y')dy'dx \, dr dy \nu(dz) \\
\nonumber
& \les m_{i+j} t \| \varphi_{2t,R} * \k \|_{L^1(\bR)} \\
\nonumber
& \les  m_{i+j} 2t^2 \|\k\|_{L^1(\bR)}R \\
\label{Du v2}
& \les m_{i+j} R.
\end{align}

\begin{itemize}
\item{\bf Estimation of $\gamma_1$.} Applying \eqref{Du trivial}, followed by \eqref{D2u v2} with $i = j = 1$, and then \eqref{young-phi},
\begin{align*}
\gamma_1^p & \les \frac{1}{R^p} \int_Z \left( \int_Z \|D_{\xi_2}F_R(t) \|_{2p}
\|D^2_{\xi_1,\xi_2} F_R(t) \|_{2p}\fm(d\xi_2)\right)^{p}\fm(d\xi_1)\\
& \les \frac{1}{R^p} \int_{\bR_0}\int_{\bR} \int_{0}^t \left( \int_{\bR_0} \int_{\bR} \int_{0}^t |z_2|
\|D^2_{(r_1,y_1,z_1),(r_2,y_2,z_2)} F_R(t) \|_{2p} \, dr_2 dy_2 \nu(dz_2)\right)^{p} \, dr_1 dy_1 \nu(dz_1)\\
& \les \frac{1}{R^p}  \int_{\bR_0} \int_{\bR} \int_{0}^t \left(m_2|z_1| (\varphi_{2t,R} * \k)(y_1) \right)^p dr_1 dy_1 \nu(dz_1)\\
& = \frac{1}{R^p} m_2^p m_p t \| \varphi_{2t,R} * \k \|_{L^p(\bR)}^p \\
& \les \frac{R}{R^p} m_2^p m_p\| \k\|_{L^1(\bR)}^p \\ 
& \les R^{1-p}.
\end{align*}

\item {\bf Estimation of $\gamma_2$.} Applying \eqref{D2u v2} with $i = 0$ and $j = 2$, followed by \eqref{young-phi},
\begin{align*}
\gamma_2^p & \les \frac{1}{R^p} \int_Z \left( \int_Z
\|D^2_{\xi_1,\xi_2} F_R(t) \|_{2p}^2\fm(d\xi_2)\right)^{p}\fm(d\xi_1)\\
& \les \frac{1}{R^p}  \int_{\bR_0} \int_{\bR} \int_{0}^t \left(m_2|z_1|^2 (\varphi_{2t,R} * \k)(y_1) \right)^p dr_1 dy_1 \nu(dz_1)\\
& = \frac{1}{R^p} m_2^p m_{2p} t \| \varphi_{2t,R} * \k \|_{L^p(\bR)}^p \\
& \les \frac{R}{R^p} m_2^p m_{2p} \| \k\|_{L^1(\bR)}^p \\ 
& \les R^{1-p}.
\end{align*}

\item {\bf Estimation of $\gamma_3$.} Applying \eqref{Du v2} with $i=0$ and $j=q+1$,
\begin{align*}
\gamma_3 & \les \frac{1}{R^{\frac{q+1}{2}}} \int_Z \|D_{\xi}F_R(t) \|^{q+1}_{q+1} \fm(d\xi)\\
& \les m_{q+1} R^{\frac{1-q}{2}}.
\end{align*}

We take
\begin{equation}
\label{def-q}
q=\left\{
\begin{array}{ll} 
2p-1 & \mbox{if $p \in (1,\frac{3}{2}]$} \\
2 & \mbox{if $p \in (\frac{3}{2},2]$}
\end{array} \right. .
\end{equation}
Then
\[
m_{q+1}=\left\{
\begin{array}{ll} 
m_{2p} & \mbox{if $p \in (1,\frac{3}{2}]$} \\
m_3 & \mbox{if $p \in (\frac{3}{2},2]$}
\end{array} \right. <\infty \quad \mbox{and} \quad 
R^{\frac{1-q}{2}}=\left\{
\begin{array}{ll} 
R^{1-p} & \mbox{if $p \in (1,\frac{3}{2}]$} \\
R^{-\frac{1}{2}} & \mbox{if $p \in (\frac{3}{2},2]$}
\end{array} \right. \leq  R^{\frac{1-p}{p}}.
\]

Hence, 
\[
\gamma_3 \les R^{-(1-\frac{1}{p})}.
\]

\item {\bf Estimation of $\gamma_4$.} Applying \eqref{Du v2} with $i=0$ and $j=2p$,
\begin{align*}
\gamma_4^{p} & \les \frac{1}{R^p} \int_Z \|D_{\xi}F_R(t) \|^{2p}_{2p} \fm(d\xi)\\
& \les m_{2p}R^{1-p} \\
& \les R^{1-p}.
\end{align*}

\item {\bf Estimation of $\gamma_5$.} Applying \eqref{D2u v2} with $i = 0$ and $j = 2p$, 
\begin{align*}
\gamma_5^p & \les \frac{1}{R^p} \int_{Z^2}
\|D^2_{\xi_1,\xi_2} F_R(t) \|^{2p}_{2p}\fm(d\xi_2)\fm(d\xi_1)\\
& \les \frac{1}{R^p}  \int_{\bR_0} \int_{\bR} \int_{0}^t m_{2p}|z_1|^{2p} (\varphi_{2t,R} * \k)(y_1) dr_1 dy_1 \nu(dz_1)\\
& \les \frac{1}{R^p} m_{2p}^2 t  \| \varphi_{2t,R} * \k \|_{L^{1}(\bR)} \\
& \les \frac{R}{R^p} m_{2p}^2  \|\k\|_{L^1(\bR)} \\
& \les R^{1-p}.
\end{align*}

\item {\bf Estimation of $\gamma_6$.} Applying \eqref{Du trivial}, followed by the symmetric analogue of \eqref{D2u v2} with $i = j = p$, for the integral with respect to $(r_1,y_1,z_1)$, and finally \eqref{young-phi},
\begin{align*}
  \gamma_6^p & \les \frac{1}{R^p}  \int_{Z^2} \| D^2_{\xi_1,\xi_2} F_R(t) \|_{2p}^{p}
 \| D_{\xi_1}F_R(t) \|_{2p}^p
\, \fm(d\xi_1)\fm(d\xi_2) \\
& \les \frac{1}{R^p} \int_{\bR_0} \int_{\bR} \int_0^t \int_{\bR_0} \int_{\bR} \int_0^t |z_1|^p \| D^2_{(r_2,y_2,z_2),(r_1,y_1,z_1)}  F_R(t) \|_{2p}^{p}
\, dr_1 dy_1 \nu(dz_1) dr_2 dy_2 \nu(dz_2) \\
& \les \frac{1}{R^p}  \int_{\bR_0} \int_{\bR} \int_{0}^t m_{2p}|z_2|^{p} (\varphi_{2t,R}*\k)(y_2) dr_2 dy_2 \nu(dz_2) \\
& = \frac{1}{R^p} m_{2p}m_p t \| \varphi_{2t,R} * \k \|_{L^{1}(\bR)} \\
& = \frac{R}{R^p} m_{2p}m_p  \|\k\|_{L^1(\bR)} \\
& \les R^{1-p}.
\end{align*}

\item {\bf Estimation of $\gamma_7$.} Applying \eqref{Du trivial}, followed by the symmetric analogue of \eqref{D2u v2} with $i = j = 1$, for the integral with respect to $(r_1,y_1,z_1)$, and finally \eqref{young-phi},
\begin{align*}
  \gamma_7^p & \les \frac{1}{R^p}  \int_{Z^2} \| D^2_{\xi_1,\xi_2}  F_R(t) \|_{2p}
 \| D_{\xi_1}F_R(t) \|_{2p}\| D_{\xi_2}F_R(t) \|_{2p}^{2(p-1)}
\, \fm(d\xi_1)\fm(d\xi_2) \\
& \les \frac{1}{R^p}  \int_{\bR_0} \int_{\bR} \int_{0}^t |z_2|^{2(p-1)}\int_{\bR_0} \int_{\bR} \int_{0}^t|z_1| \| D^2_{(r_2,y_2,z_2),(r_1,y_1,z_1)}  F_R(t) \|_{2p}
\, dr_1 dy_1 \nu(dz_1) dr_2 dy_2 \nu(dz_2) \\
& \les \frac{1}{R^p}  \int_{\bR_0} \int_{\bR} \int_{0}^t m_{2}|z_2|^{2p-1} ( \varphi_{2t,R} * \k)(y_2) dr_2 dy_2 \nu(dz_2) \\
& = \frac{1}{R^p} m_2m_{2p-1} t \| \varphi_{2t,R} * \k \|_{L^1(\bR)} \\
& = \frac{R}{R^p} m_2m_{2p-1} 2t^2 \|\k\|_{L^1(\bR)} \\
& \les  R^{1-p}.
\end{align*}
For the last line, since $p \in (1,2]$, we have $p < 2p-1 < 2p$ and because $m_p,m_{2p} < \infty$ we conclude that $m_{2p-1} < \infty$.
\end{itemize}

\medskip

{\bf Case 2.} Assume that \fbox{$k=R_{1,\alpha/2}$} for some $\alpha \in (0,1)$. By Theorem \ref{cov-th}, $\sigma_R^2(t) \sim  R^{\alpha+1}$. 
%Let $p \in (1,2]$ be such that $m_p<\infty$ and $m_{2p}<\infty$.
Recall that in this case, we assume that 
$p>\frac{2}{2-\alpha}$.

\begin{itemize}

\item {\bf Estimation of $\gamma_1$.} Applying \eqref{Du} and \eqref{D2u-1}, we have:
\begin{align*}
\gamma_1^p & \les \frac{1}{R^{(\alpha+1)p}} \int_Z \left( \int_Z \|D_{\xi_2}F_R(t) \|_{2p}
\|D^2_{\xi_1,\xi_2} F_R(t) \|_{2p}\fm(d\xi_2)\right)^{p}\fm(d\xi_1)\\
& \les \frac{1}{R^{(\alpha+1)p}} \int_{\bR_0}\int_{\bR}  \Bigg[ \int_{\bR_0} \int_{\bR}  \left(|z_2| \int_{-R}^R \int_{\bR} G_t(x_2-y_2'') \k(y_2-y_2'') dy_2'' dx_2\right) \\
& \quad \quad
\left(|z_1 z_2| \int_{-R}^R \int_{\bR^2} G_t(x_1-y_1')G_t(y_1'-y_2') \k(y_1-y_1')\k(y_2-y_2') dy_1' dy_2' dx_1 \right) \\
& \qquad \qquad \qquad \qquad \qquad  dy_2 \nu(dz_2)\Bigg]^{p} \, dy_1 \nu(dz_1)\\
& \les \frac{m_2^p m_p}{R^{(\alpha+1)p}} \int_{\bR} \Bigg[ \int_{\bR^4} \int_{[-R,R]^2} G_t(x_1-y_1') G_t(y_1'-y_2') G_t(x_2-y_2'') \\
& \quad \quad \quad \quad \k(y_1-y_1') \k(y_2-y_2') \k(y_2-y_2'') dx_1 dx_2 dy_1' dy_2' dy_2''  dy_2\Bigg]^p dy_1\\
& = \frac{m_2^p m_pt^{p+1}}{R^{(\alpha+1)p}} \int_{\bR} \Bigg[ \int_{\bR^3} \varphi_{t,R}(y_1') G_t(y_1'-y_2') \varphi_{t,R}(y_2'') \k(y_1-y_1') f(y_2'-y_2'') dy_1' dy_2' dy_2'' \Bigg]^p dy_1,
\end{align*}
where for the last line we used definition \eqref{def-phi-tR} of $\varphi_{t,R}$ and the fact that $f=\k *\k$.

We now apply Theorem \ref{HLS} to $\k=R_{1,\alpha/2}$. For any $\varphi \in L^{q_1}(\bR)$, we have:
\begin{equation}
\label{HLS-gam1}
\int_{\bR} \left( \int_{\bR} \varphi(y')\k(y-y')dy'\right)^p dy \leq C \left(\int_{\bR} |\varphi(y')|^{q_1}dy' \right)^{p/q_1},
\end{equation}
where $C>0$ is a constant depending on $(\alpha,p)$, with

\begin{equation}
\label{def-q1}
\frac{1}{q_1}=\frac{1}{p}+\frac{\alpha}{2} \quad \mbox{and} \quad p>\frac{2}{2-\alpha}.
\end{equation} 

We apply this inequality to the function
\[
\varphi(y_1')= \varphi_{t,R}(y_1')\int_{\bR^2} G_t(y_1'-y_2') \varphi_{t,R}(y_2'')  f(y_2'-y_2'') dy_2' dy_2''.
\]
We obtain that
\begin{align*}
\gamma_1^p & \les 
%\frac{1}{R^{(\alpha+1)p}} \Bigg[ \int_{\bR} \left( \int_{\bR^2} \varphi_{t,R}(y_1') G_t(y_1'-y_2') %\varphi_{t,R}(y_2'')  f(y_2'-y_2'') dy_2' dy_2'' \right)^q dy_1' \Bigg]^{p/q}\\
 \frac{1}{R^{(\alpha+1)p}} \Bigg[ \int_{\bR} \varphi_{t,R}^{q_1}(y_1') \left( \int_{\bR^2}  G_t(y_1'-y_2') \varphi_{t,R}(y_2'')  f(y_2'-y_2'') dy_2' dy_2'' \right)^{q_1} dy_1' \Bigg]^{p/q_1}.
\end{align*}
For the inner integral, we apply Lemma \ref{Hol-HLS} to $f=R_{1,\alpha}$. Let

\begin{equation}
\label{def-rs}
\frac{1}{r}+\frac{1}{s}=1+\alpha \quad \mbox{and} \quad 1 < r <\frac{1}{\alpha}.
\end{equation}

Then
\[
\int_{\bR^2}  G_t(y_1'-y_2') \varphi_{t,R}(y_2'')  f(y_2'-y_2'') dy_2' dy_2''  \leq C \|G_t\|_{L^r(\bR)}\|\varphi_{t,R}\|_{L^s(\bR)} \les R^{1/s}
\]
where for the last inequality we used \eqref{norm-phi}, and the fact that $\|G_t\|_{L^r(\bR)}=t^{1/r}$.
Using the above and \eqref{norm-phi} again,  we infer that
\begin{align*}
\gamma_1^p & \les \frac{1}{R^{(\alpha+1)p}} R^{p/s} \left( \int_{\bR}\varphi_{t,R}^{q_1}(y_1') dy_1' \right)^{p/q_1} \les \frac{1}{R^{(\alpha+1)p}} R^{p/s}  \|\varphi_{t,R}\|_{L^q(\bR)}^p \les \frac{1}{R^{(\alpha+1)p}} R^{p/s} R^{p/q_1}.
\end{align*}
Hence,
\[
\gamma_1^p \les R^{-p\big(\alpha+1-\frac{1}{s}-\frac{1}{q_1}\big)}.
\]
To calculate the exponent, we use the definitions of $s$ and $q_1$:
\[
\alpha+1-\frac{1}{s}-\frac{1}{q_1}=\alpha+1-\left(1+\alpha-\frac{1}{r} \right)-\left(\frac{1}{p}+\frac{\alpha}{2} \right)=\frac{1}{r}-\frac{1}{p}-\frac{\alpha}{2}.
\]

We  choose $r>1$ such that $\frac{1}{r}>\max\big(\alpha,\frac{1}{p}+\frac{\alpha}{2}\big)=\frac{1}{p}+\frac{\alpha}{2}=\frac{2+\alpha p}{2p}$. We obtain that
\[
\gamma_1 \les R^{-\big(\frac{1}{r}-\frac{1}{p}-\frac{\alpha}{2} \big)} \quad \mbox{for any $r \in \big(1, \frac{2p}{2+\alpha p}\big)$.}
\]
This range of $r$ is non-empty since $p>\frac{2}{2-\alpha}$.

\item {\bf Estimation of $\gamma_2$.} Using \eqref{D2u}, we have:
\begin{align*} 
\gamma_2^p & \les \frac{1}{R^{(\alpha+1)p}} \int_{Z} \left( \int_{Z} \|D_{\xi_1,\xi_2} F_R(t)\|_{p}^2 \fm(d\xi_2) \right)^p \fm(d\xi_1) \\
& \les  \frac{1}{R^{(\alpha+1)p}}  \int_{\bR_0}\int_{\bR} \Bigg[\int_{\bR} \int_{\bR_0} |z_1 z_2|^2 \Big(\int_{-R}^R \int_{\bR^2}
G_{2t}(x-y_1') G_{2t}(x-y_2') \k(y_1-y_1') \\
& \qquad \qquad \quad \qquad \qquad \k(y_2-y_2') dy_1'dy_2' dx \Big)^2 dy_2 \nu(dz_2)  \Bigg]^p dy_1 \nu(dz_1)  \\
& \les \frac{m_2^p m_{2p}}{R^{(\alpha+1)p}} \int_{\bR} \Bigg[ \int_{\bR} \int_{[-R,R]^2} \int_{\bR^4}
G_{2t}(x_1-y_1') G_{2t}(x_1-y_2')\k(y_1-y_1') \k(y_2-y_2')\\
& \quad G_{2t}(x_2-y_1'') G_{2t}(x_2-y_2'')\k(y_1-y_1'') \k(y_2-y_2'')dy_1'dy_2' dy_1'' dy_2'' dx_1 dx_2 dy_2 
\Bigg]^p dy_1 \\
& = C \frac{1}{R^{(\alpha+1)p}}  \int_{\bR} \Bigg[  \int_{[-R,R]^2} \int_{\bR^2} G_{2t}(x_1-y_1') G_{2t}(x_2-y_1'')
\k(y_1-y_1') \\
& \quad \k(y_1-y_1'') \Big(\int_{\bR^2} G_{2t}(x_1-y_2') G_{2t}(x_2-y_2'') f(y_2'-y_2'')dy_2' dy_2'' \Big)dy_1'  dy_1'' dx_1 dx_2 
\Bigg]^p dy_1.
\end{align*}
For the inner integral, we apply Lemma \ref{Hol-HLS} to $f=R_{1,\alpha}$, with $r,s$ given by \eqref{def-rs}:
\[
\int_{\bR^2} G_{2t}(x_1-y_2') G_{2t}(x_2-y_2'') f(y_2'-y_2'')dy_2' dy_2'' \leq C \|G_{2t}\|_{L^r(\bR)} \|G_{2t}\|_{L^s(\bR)} \leq C.
\]
Hence,
\begin{align*}
\gamma_2^p & \les \frac{1}{R^{(\alpha+1)p}} \int_{\bR} \Bigg(  \int_{-R}^R \int_{\bR} G_{2t}(x-y_1')
\k(y_1-y_1')dy_1'  dx \Bigg)^{2p} dy_1\\
& =C  \frac{1}{R^{(\alpha+1)p}}  \|\varphi_{2t,R}*\k\|_{L^{2p}(\bR)}^{2p} \les \frac{1}{R^{(\alpha+1)p}}  \|\varphi_{2t,R}\|_{L^{q_2}(\bR)}^{2p},
\end{align*}
where for the last inequality we applied Theorem \ref{HLS} to $\k=R_{1,\alpha/2}$, with
\begin{equation}
\label{def-q2}
\frac{1}{q_2}=\frac{1}{2p}+\frac{\alpha}{2} \quad \mbox{and} \quad p>\frac{1}{2-\alpha}.
\end{equation}
Using \eqref{norm-phi}, we obtain that 
$\gamma_2^p \les \frac{1}{R^{(\alpha+1)p}} R^{\frac{2p}{q_2}}=C R^{1-p}$. Hence,
\[
\gamma_p \les R^{-(1-\frac{1}{p})}.
\]

\item {\bf Estimation of $\gamma_3$.} Let $q \in (1,2]$ be arbitrary. Then
\begin{align*}
\gamma_3 & \les \frac{1}{R^{\frac{1}{2}(\alpha+1)(q+1)}}  \int_{Z} \|D_{\xi_1,\xi_2} F_R(t)\|_{q+1}^{q+1} \fm(d\xi)\\
& \les  \frac{1}{R^{\frac{1}{2}(\alpha+1)(q+1)}} \int_{\bR} \int_{\bR_0} |z|^{q+1} \big[\big(\varphi_{t,R}* \k\big)(y)\big]^{q+1} \nu(dz)dy\\
&=C \frac{1}{R^{\frac{1}{2}(\alpha+1)(q+1)}} m_{q+1} \|\varphi_{t,R}* \k\|_{L^{q+1}(\bR)}^{q+1} \les \frac{1}{R^{\frac{1}{2}(\alpha+1)(q+1)}}  \|\varphi_{t,R}\|_{L^{q'}(\bR)}^{q+1},
\end{align*}
where for the last inequality we used Theorem \ref{HLS} with
\[
\frac{1}{q'}=\frac{1}{q+1}+\frac{\alpha}{2} \quad \mbox{and} \quad q +1 > \frac{2}{2-\alpha}.
\]
Note that the last condition holds since $q>1>\frac{\alpha}{2-\alpha}$. Using \eqref{norm-phi}, we have:
\[
\gamma_3 \les \frac{1}{R^{\frac{1}{2}(\alpha+1)(q+1)}}   R^{\frac{q+1}{q'}}=CR^{\frac{1-q}{2}}.
\]
Choosing $q$ as in \eqref{def-q}, we have:
\[
\gamma_3 \les R^{-(1-\frac{1}{p})}.
\]

\item {\bf Estimation of $\gamma_4$.} Similarly to the estimate of $\gamma_3$ above, but with $q+1$ replaced by $2p$, we have:
\begin{align*}
\gamma_4^p & \les \frac{1}{R^{(\alpha+1)p}}\int_{Z}\|D_{\xi} F_R(t)\|_{2p}^{2p} \fm(d\xi) \\
& \les \frac{1}{R^{(\alpha+1)p}}  \|\varphi_{t,R}\|_{L^{q_2}(\bR)}^{2p} \les \frac{1}{R^{(\alpha+1)p}} R^{\frac{2p}{q_2}} =C R^{1-p},
\end{align*}
where
\begin{equation}
\label{def-q2}
\frac{1}{q_2}=\frac{1}{2p}+\frac{\alpha}{2} \quad \mbox{and} \quad p > \frac{1}{2-\alpha}.
\end{equation}
Hence,
\[
\gamma_4 \les R^{-(1-\frac{1}{p})}.
\]

\item {\bf Estimation of $\gamma_5$.} Using \eqref{D2u}, we have:
\begin{align*}
\gamma_5^p & \les  \frac{1}{R^{(\alpha+1)p}}\int_{Z^2} \|D_{\xi_1,\xi_2}F_R(t)\|_{2p}^{2p}\fm(d\xi_1) \fm(d\xi_2) \\
& \les \frac{1}{R^{(\alpha+1)p}} \int_{\bR_0^2} \int_{\bR^2} |z_1 z_2|^{2p}   \Bigg(
\int_{-R}^R \int_{\bR^2} G_t(x-y_1') G_t(x-y_2')\k(y_1-y_1')\\
& \qquad \qquad \qquad \qquad \qquad \qquad \k(y_2-y_2')  dy_1' dy_2' dx\Bigg)^{2p}  dy_1 dy_2 \nu(dz_1) \nu(dz_2)\\
& \les \frac{m_{2p}^2}{R^{(\alpha+1)p}}  \int_{\bR^2}  \Bigg[\int_{\bR}
 \Big( \int_{-R}^R \int_{\bR} G_t(x-y_1') G_t(x-y_2')\k(y_1-y_1') dy_1' dx \Big) \\
 & \qquad \qquad \qquad \qquad \qquad \k(y_2-y_2')  dy_2' \Bigg]^{2p}  dy_2 dy_1.
\end{align*}

We now apply Theorem \ref{HLS} to $\k=R_{1,\alpha/2}$. For any $\varphi \in L^{q_2}(\bR)$, we have:
\[
\int_{\bR}\left( \int_{\bR}\varphi(y)\k(y-y') dy'\right)^{2p} dy \leq C \left(\int_{\bR}|\varphi(y')|^{q_2}dy' \right)^{2p/q_2},
\]
where $q_2$ is given by \eqref{def-q2}. We apply this inequality to the function:
\[
\varphi_{y_1}(y_2')=\int_{-R}^R \int_{\bR} G_t(x-y_1') G_t(x-y_2')\k(y_1-y_1') dy_1' dx 
\]
for any $y_1\in \bR$ fixed. We obtain that:
\begin{align*}
\gamma_5^p & \les \frac{1}{R^{(\alpha+1)p}} \int_{\bR} \Bigg[ \int_{\bR} \Big( \int_{-R}^R \int_{\bR} G_t(x-y_1') G_t(x-y_2')\k(y_1-y_1') dy_1' dx \Big)^{q_2} dy_2' \Bigg]^{2p/q_2} dy_1.
\end{align*}

For the inner integral, we use H\"older's inequality with respect to the finite measure $\mu(dx)=G_t(x-y_2')dx$ on $[-R,R]$ (whose total mass is bounded by $t$):
\begin{align}
\nonumber
& \int_{\bR} \Big( \int_{-R}^R \int_{\bR} G_t(x-y_1') G_t(x-y_2')\k(y_1-y_1') dy_1' dx \Big)^{q_2} dy_2' \\
\nonumber
& \quad \leq  t^{q_2-1} \int_{\bR}  \int_{-R}^R \Big(\int_{\bR} G_t(x-y_1') \k (y_1-y_1') dy_1' \Big)^{q_2} G_t(x-y_2')dx dy_2'\\
\label{gam5-1}
&\quad = t^{q_2} \int_{-R}^R \Big(\int_{\bR} G_t(x-y_1') \k (y_1-y_1') dy_1' \Big)^{q_2} dx,
\end{align}
and for the last line we used the fact that $\int_{\bR}G_t(x-y_2')dy_2'=t$. Hence,
\begin{align*}
\gamma_5^p & \les \frac{1}{R^{(\alpha+1)p}} \int_{\bR} \Bigg[\int_{-R}^R \Big(\int_{\bR} G_t(x-y_1') \k (y_1-y_1') dy_1' \Big)^{q_2} dx \Bigg]^{2p/q_2} dy_1.
\end{align*}
We use again H\"{o}lder's inequality with respect to $dx$ measure on $[-R,R]$. We obtain:
\begin{align*}
\gamma_5^p & \les \frac{1}{R^{(\alpha+1)p}} (2R)^{\frac{2p}{q_2}-1} \int_{\bR} \int_{-R}^R \Big(\int_{\bR} G_t(x-y_1') \k (y_1-y_1') dy_1' \Big)^{2p} dx dy_1 \\
& =C \frac{1}{R^{(\alpha+1)p}} R^{\frac{2p}{q_2}-1} \int_{-R}^R \|G_t(x-\cdot)*\k\|_{L^{2p}(\bR)}^{2p} dx.
\end{align*}
We apply Theorem \ref{HLS} to $k=R_{1,\alpha/2}$, with $q_2$ given by \eqref{def-q2}. We obtain that:
\[
\|G_t(x-\cdot)*\k\|_{L^{2p}(\bR)}^{2p}\leq C\|G_{t}(x-\cdot)\|_{L^{q_2}(\bR)}^{2p}\leq C.
\]
Hence,
\[
\gamma_5^p \les \frac{1}{R^{(\alpha+1)p}} R^{\frac{2p}{q_2}-1} (2R)=C R^{-p(\alpha+1)+\frac{2p}{q_2}}=CR^{1-p},
\]
and 
\[
\gamma_5 \les R^{-(1-\frac{1}{p})}.
\]

\item {\bf Estimation of $\gamma_6$.} Using \eqref{Du} and \eqref{D2u}, we have:
\begin{align*}
\gamma_6^p & \les \frac{1}{R^{(\alpha+1)p}} \int_{Z^2} \|D_{\xi_1} F_R(t) \|_{2p}^p  \|D_{\xi_1,\xi_2} F_R(t)\|_{2p}^p \fm(d\xi_1) \fm(d\xi_2)\\
& \les \frac{1}{R^{(\alpha+1)p}} \int_{\bR_0^2} \int_{\bR^2} |z_1|^{2p} |z_2|^p \big[\big(\varphi_{t,R}*k\big)(y_1)\big]^p \\
& \quad \Big( \int_{-R}^R \int_{\bR^2} G_t(x-y_1')G_t(x-y_2') \k(y_1-y_1') \k(y_2-y_2') dy_1' dy_2' dx \Big)^p dy_1 dy_2 \nu(dz_1) \nu(dz_2)\\
& \les \frac{m_{2p} m_p}{R^{(\alpha+1)p}} \int_{\bR} \big[\big(\varphi_{t,R}*k\big)(y_1)\big]^p\\
& \quad \int_{\bR}  \Big( \int_{-R}^R \int_{\bR^2} G_t(x-y_1')G_t(x-y_2') \k(y_1-y_1') \k(y_2-y_2') dy_1' dy_2' dx \Big)^p dy_2 dy_1.
\end{align*}
We use H\"older's inequality $\langle \phi, \psi \rangle_{L^2(\bR)} \leq \|\phi\|_{L^a(\bR)} \|\psi\|_{L^b(\bR)}$ with $a,b>1$ and
\[
\frac{1}{a}+\frac{1}{b}=1,
\]
for the functions $\phi(y_1)=\big[ \big(\varphi_{t,R}*k\big)(y_1)\big]^p$ and
\[
\psi(y_1)= \int_{\bR}  \Big( \int_{-R}^R \int_{\bR^2} G_t(x-y_1')G_t(x-y_2') \k(y_1-y_1') \k(y_2-y_2') dy_1' dy_2' dx \Big)^p dy_2.
\]
Note that 
\[
\|(\varphi_{t,R}* \k)^p\|_{L^a(\bR)}=\|\varphi_{t,R}* \k\|_{L^{ap}(\bR)}^p \leq C \|\varphi_{t,R}\|_{L^{q_3}(\bR)}^p \les R^{p/q_3},
\]
by applying Theorem \ref{HLS} with 
\begin{equation}
\label{def-q3}
\frac{1}{q_3}=\frac{1}{ap}+\frac{\alpha}{2} \quad \mbox{and} \quad ap>p>\frac{2}{2-\alpha}.
\end{equation}

We obtain:
\begin{align*}
\gamma_6^p & \les \frac{1}{R^{(\alpha+1)p}} R^{p/q_3}\Bigg\{\int_{\bR}  \Bigg[ \int_{\bR}  \Big( \int_{-R}^R \int_{\bR^2} G_t(x-y_1')G_t(x-y_2') \k(y_1-y_1') \\
& \qquad \qquad \qquad \qquad \qquad \qquad \k(y_2-y_2') dy_1' dy_2' dx  \Big)^p dy_2 \Bigg]^b dy_1 \Bigg\}^{1/b}.
\end{align*}

We apply inequality \eqref{HLS-gam1} for the function $\varphi_{y_1,y_2}$ given by:
\[
\varphi_{y_1,y_2}(y_2')=\int_{-R}^R\int_{\bR}G_t(x-y_1')G_t(x-y_2') \k(y_1-y_1')  dy_1'dx.
\] 
We obtain:
\begin{align*}
\gamma_6^p & \les \frac{1}{R^{(\alpha+1)p}} R^{p/q_3} \Bigg\{\int_{\bR} \Bigg[\int_{\bR} \Big(\int_{-R}^R\int_{\bR}G_t(x-y_1')G_t(x-y_2') \k(y_1-y_1')  dy_1'dx\Big)^{q_1} dy_2'\Bigg]^{bp/q_1} dy_1 \Bigg\}^{1/b} \\
& \les \frac{1}{R^{(\alpha+1)p}} R^{p/q_3} \Bigg\{\int_{\bR} \Bigg[ \int_{-R}^R \Big( \int_{\bR} G_t(x-y_1') \k(y_1-y_1') dy_1' \Big)^{q_1} dx\Bigg]^{bp/q_1} dy_1 \Bigg\}^{1/b} 
\end{align*}
where for the last line, we used \eqref{gam5-1} with $q_2$ replaced by $q_1$. We apply H\"older's inequality for the $dx$ measure on $[-R,R]$. We obtain:
\begin{align*}
\gamma_6^p & \les \frac{1}{R^{(\alpha+1)p}} R^{p/q_3} \Bigg\{ (2R)^{\frac{bp}{q_1}-1} \int_{\bR}\int_{-R}^R 
\Big( \int_{\bR}  G_t(x-y_1') \k(y_1-y_1') dy_1' \Big)^{bp} dx dy_1 \Bigg\}^{1/b}\\
&=C \frac{1}{R^{(\alpha+1)p}} R^{p/q_3} R^{\frac{p}{q_1}-\frac{1}{b}}
\Bigg( \int_{-R}^R  \|G_{t}(x-\cdot)*\k \|_{L^{bp}(\bR)}^{bp} dx \Bigg)^{1/b}.
\end{align*}
We apply Theorem \eqref{HLS} to $\k=R_{1,\alpha/2}$, with
\begin{equation}
\frac{1}{q_4}=\frac{1}{bp}+\frac{\alpha}{2} \quad \mbox{and} \quad bp>p>\frac{2}{2-\alpha},
\end{equation}
to obtain that $\|G_{t}(x-\cdot)*\k \|_{L^{bp}(\bR)}^{bp} \leq C \|G_t(x-\cdot)\|_{L^{q_4}(\bR)}^{bp} \leq C$. Hence,
\begin{align*}
\gamma_6^p \les  \frac{1}{R^{(\alpha+1)p}} R^{p/q_3} R^{\frac{p}{q_1}-\frac{1}{b}} (2R)^{1/b}.
\end{align*}
We compute the exponent of $R$, using the definitions of $q_3$ and $q_1$:
\[
-(\alpha+1)p+\frac{p}{q_3} +\frac{p}{q_1}=-(\alpha+1)p+\left(\frac{1}{a}+\frac{\alpha p}{2}\right)+\left(1+\frac{\alpha p}{2} \right)=-\left(p-1-\frac{1}{a}\right).
\]
This exponent is negative if $a>\frac{1}{p-1}$.
Hence,
\[
\gamma_6 \leq R^{-(1-\frac{1}{p}-\frac{1}{ap})} \quad \mbox{for any $a>\frac{1}{p-1}$}.
\]

\item {\bf Estimation of $\gamma_7$.} Applying \eqref{Du} and \eqref{D2u}, we have:
\begin{align*}
\gamma_7^p & \les \frac{1}{R^{(\alpha+1)p}} \int_{Z^2} 
\|D^2_{\xi_1,\xi_2} F_R(t) \|_{2p}\|D_{\xi_1}F_R(t) \|_{2p}\|D_{\xi_2}F_R(t) \|^{2(p-1)}_{2p}\fm(d\xi_2)\fm(d\xi_1)\\
& \les \frac{1}{R^{(\alpha+1)p}} \int_{\bR_0^2}\int_{\bR^2}   \left(|z_1| \int_{-R}^R \int_{\bR} G_t(x_2-y_1'') \k(y_1-y_1'') dy_1'' dx_2\right) \\
& \qquad
\left(|z_1 z_2| \int_{-R}^R \int_{\bR^2} G_{2t}(x_1-y_1')G_{2t}(x_1-y_2') \k(y_1-y_1')\k(y_2-y_2') dy_1' dy_2' dx_1 \right) \\
& \qquad 
\left(|z_2| \int_{-R}^R \int_{\bR} G_t(x_3-y_2'') \k(y_2-y_2'') dy_2'' dx_3\right)^{2(p-1)} dy_2 \nu(dz_2) dy_1 \nu(dz_1)\\
& = \frac{m_2m_{2p-1}}{R^{(\alpha+1)p}} \int_{\bR^2} \int_{-R}^{R} \left( \int_{\bR^2} G_{2t}(x_1-y_1') \varphi_{t,R}(y_1'') f(y_1'-y_1'')   dy_1' dy_1'' \right)\\
& \qquad  G_{2t}(x_1-y_2')\k(y_2-y_2') \big[ (\varphi_{t,R} *k)(y_2)\big]^{2(p-1)} dx_1 dy_2' dy_2,
\end{align*}
where for the last line we used definition \eqref{def-phi-tR} of $\varphi_{t,R}$ and the fact that $f=\k *\k$. 

We apply Lemma \ref{Hol-HLS} to $f = R_{1,\alpha}$, with $r,s>1$ given by \eqref{def-rs}:
\[
\int_{\bR^2}  G_{2t}(x_1-y_1') \varphi_{t,R}(y_1'') f(y_1'-y_1'') dy_1' dy_1''  \les \|G_{2t}(x_1-\cdot)\|_{L^{s}(\bR)}\|\varphi_{t,R}(y_1')\|_{L^{r}(\bR)}  \les R^{1/r},
\]
where for the last inequality we used \eqref{def-rs}. Therefore, 
\begin{align*}
\gamma_7^p & \les \frac{1}{R^{(\alpha+1)p}}R^{1/r} \int_{\bR^2} \int_{-R}^{R} G_{2t}(x_1-y_2')\k(y_2-y_2') \big[ (\varphi_{t,R}*k)(y_2)\big]^{2(p-1)} dx_1 dy_2' dy_2 
\\
& \les \frac{1}{R^{(\alpha+1)p}}R^{1/r} \| \varphi_{2t,R}*k\|_{L^{2p-1}(\bR)}^{2p-1}.
\end{align*}
Applying Theorem \ref{HLS} to $\k = R_{1,\alpha/2}$ with 
\begin{equation}
\label{def-q5}
\frac{1}{q_5}=\frac{1}{2p-1}+\frac{\alpha}{2} \quad \mbox{and} \quad 2p-1>p>\frac{2}{2-\alpha},
\end{equation}
we obtain:
\begin{equation*}
\gamma_7^p \les \frac{1}{R^{(\alpha+1)p}}R^{1/r} \| \varphi_{2t,R}\|_{L^{q_5}}^{2p-1} \les \frac{1}{R^{(\alpha+1)p}}R^{1/r} R^{\frac{2p-1}{q_5}}.
\end{equation*}
We use \eqref{def-q5} to compute the exponent, 
\begin{align*}
-(\alpha+1)p + \frac{1}{r} + \frac{2p-1}{q_5}=
& = -\alpha p-p + \frac{1}{r} + 1 + \frac{(2p-1)\alpha}{2} 
\\
& = -\left(p+ \frac{\alpha}{2} -1 -\frac{1}{r}\right).
\end{align*}
The exponent is negative when $\frac{1}{r} < p + \frac{\alpha}{2} - 1$. By definition \eqref{def-rs} we also require that $\alpha < \frac1r < 1$. Therefore, we choose $\frac{1}{r} \in (\alpha, \min \{1, p + \frac{\alpha}{2} - 1\})$. This interval is non-empty since $p>\frac{2}{2-\alpha}>\frac{2+\alpha}{2}$, and hence, $p+\frac{\alpha}{2}-1>\alpha$.
%\[ p + \frac{\alpha}{2} - 1 =  p - \frac{\alpha+2}{2}  + \alpha > \frac{2}{2-\alpha} - \frac{\alpha + 2}{2} + %\alpha = \frac{\alpha^2}{2(2-\alpha)} + \alpha > \alpha.  \]

In summary, $\gamma_i \les R^{-a_i}$ for any $i \in \{1,\ldots,7\}$, where $a_i>0$ is given by:
\[
a_i=
\left\{
\begin{array}{ll} 
\frac{1}{r}-\frac{1}{p}-\frac{\alpha}{2} & \mbox{for any $r \in (1,\frac{2p}{2+\alpha p})$, if $i=1$, } \\
1-\frac{1}{p} & \mbox{if $i \in \{2,3,4,5\}$,} \\
1-\frac{1}{p} -\frac{1}{ap} & \mbox{for any $a>\frac{1}{p-1}$, if $i=6$,}\\
1 + \frac{\alpha}{2p} - \frac{1}{p}(1+\frac{1}{r}) & \mbox{for any $r$ with $\frac{1}{r} \in (\alpha, \min \{1, p + \frac{\alpha}{2} - 1\})$, if $i=7$.}
\end{array} \right. 
\]

Note that the range of $a_1$ is $(0,1-\frac{1}{p}-\frac{\alpha}{2})$, the range of $a_6$ is $(0,1-\frac{1}{p})$, and the range of $a_7$ is $\left(\max\{0,1+\frac{\alpha}{2p}-\frac{2}{p}\}, 1-\frac{1}{p}-\frac{\alpha}{2p}\right)$. The conclusion follows since if $\gamma \leq R^{-\varepsilon_0}$ for some $\varepsilon_0>0$, then $\gamma \leq R^{-\varepsilon}$ for all $\varepsilon \in (0,\varepsilon_0)$.
\end{itemize}
%% END QCLT

%% BEGIN FCLT
\subsection{Functional CLT}

In this section, we prove Theorem \ref{FCLT}, by showing finite-dimensional convergence and tightness. We  will use Kolmogorov-Chentsov criterion (Theorem 3.23 of \cite{kallenberg02}) for the existence of the H\"older continuous modification, and the classical moment criterion (Corollary 16.9 of \cite{kallenberg02}) for tightness. Recall that by Theorem \ref{cov-th}, $\sigma_R^2(t) \sim  R^{\beta}$. We fix $T>0$. 

\medskip

{\em Step 1.} (finite-dimensional convergence) We will prove that for any integer $m\geq 1$ and form any $t_1,\ldots,t_m \in [0,T]$
\[
\Big(\frac{1}{R^{\beta/2}} F_R(t_1),\ldots, \frac{1}{R^{\beta/2}}F_R(t_m)  \Big)
\stackrel{d}{\to} \Big(\cG(t_1),\ldots, \cG(t_m) \Big) \quad \mbox{as $R \to \infty$}.
\]
By Cram\'er-Wold theorem, this is equivalent to showing that for any $b_1,\ldots,b_m \in \bR$,
\[
X_R:=\frac{1}{R^{\beta/2}}\sum_{j=1}^{m}b_jF_R(t_j) \stackrel{d}{\to}\sum_{j=1}^{m}b_j \cG(t_j) 
\quad \mbox{as $R \to \infty$}.
\]
Using the same argument as on page 4215 of \cite{BZ24}, it is enough to prove that $X_R/\tau_R \stackrel{d}{\to} Z$ as $R \to \infty$, where $\tau_R^2={\rm Var}(X_R)$ and $Z \sim N(0,1)$. By Proposition \ref{tara}, 
\[
d_{W}\left(\frac{X_R}{\tau_R},Z \right) \leq \gamma_1+\gamma_2+\gamma_3,
\]
where $\gamma_1,\gamma_2$ and $\gamma_3$ are defined as in \eqref{gamma17} with $F=X_R$. Using the same argument as in the proof of Theorem \ref{QCLT}, we infer that 
\[
d_{W}\left(\frac{X_R}{\tau_R},Z \right) \les R^{1-p} \quad \mbox{if $\k \in L^1(\bR)$,}
\]
and
\[
d_{W}\left(\frac{X_R}{\tau_R},Z \right) \les R^{-\varepsilon} \quad \mbox{for any $\varepsilon \in \big(0,1-\frac{1}{p}-\frac{\alpha}{2}\big)$, \quad if $\k =R_{1,\alpha/2}$.}
\]

This implies that $X_R/\tau_R \stackrel{d}{\to} Z$ as $R \to \infty$.

\medskip

{\em Step 2.} (H\"older continuity and tightness) 

We will prove that for any $p'\geq 2$ such that $m_{p'}<\infty$ and for any  $0\leq s<t\leq T$, $R\geq1$,
\begin{equation}
\label{incr-F}
\bE|F_R(t)-F_R(s)|^{p'} \leq C_T R^{\beta p'/2} (t-s)^{p'},
\end{equation}
where $C_T>0$ is a constant depending on $T$. In particular, 
using \eqref{incr-F} for $p'=2$ or $p'=2p$, and
the two criteria mentioned above, we infer that $\{F_R(t)\}_{t\in [0,T]}$ has a $\gamma$-H\"older continuous modification for any $\gamma \in (0,\frac{\beta}{2})$, and $\{R^{-\beta/2} F_R(\cdot)\}_{R>0}$ is tight in $C[0,T]$. Note when $\beta = 1$, inequality \eqref{incr-F} mirrors the L\'evy white noise case (see (2.65) of \cite{BZ24}). (From the proof below, we see that relation \eqref{incr-F} with $p'=2$ holds for all $R>0$.)

\medskip

To prove \eqref{incr-F}, we write
\begin{align*}
F_R(t)&=\int_{-R}^R \int_0^t \int_{\bR} G_{t-r}(x-y)u(r,y) X(dr,dy)dx 
=\int_0^t \int_{\bR}\varphi_{t,R}(r,y)u(r,y)X(dr,dy),
\end{align*}
where $\varphi_{t,R}(r,y)$ is given by \eqref{def-phi-tR}. Therefore
\begin{align*}
F_R(t)-F_R(s) &= \int_0^s \int_{\bR} \big( \varphi_{t,R}(r,y)-\varphi_{s,R}(r,y)\big)u(r,y)X(dr,dy)+\\
& \quad \int_s^t \int_{\bR}\varphi_{t,R}(r,y)u(r,y)X(dr,dy)=:T_1+T_2.
\end{align*}
We estimate the moments of $T_1$ and $T_2$ using Theorem 3.7 of \cite{BJ25} and the bound \eqref{mom-u} for $\|u(t,x)\|_{p'}$:
\begin{align*}
\bE|T_1|^{p'}  & \les \Bigg( \int_0^s  \int_{\bR^2} \big( \varphi_{t,R}(r,y)-\varphi_{s,R}(r,y)\big)
\big( \varphi_{t,R}(r,y')-\varphi_{s,R}(r,y')\big)f(y-y') dydy' dr \Bigg)^{p'/2}\\
& \quad \quad +\int_0^s \big\|\big(\varphi_{t,R}(r,\cdot)-\varphi_{s,R}(r,\cdot)\big) * \k\big\|_{L^{p'}(\bR)}^{p'} dr=:C(A+B)
\end{align*}
and
\begin{align*}
\bE|T_2|^{p'} & \les \Bigg(\int_s^t  \int_{\bR^2} \varphi_{t,R}(r,y)\varphi_{t,R}(r,y')f(y-y') dydy' dr \Bigg)^{p'/2} \\
& \quad \quad  + \int_s^t 
\|\varphi_{t,R}(r,\cdot)*  \k \|_{L^{p'}(\bR)}^{p'} dr=:C(A'+B').
\end{align*}

Each term is computed using either Young's inequality if $k \in L^1(\bR)$, or Theorem \ref{HLS} if $\k = R_{1,\alpha/2}$. Starting with $A$,  
\[
A = \Bigg( \int_0^s  \| (\varphi_{t,R}(r,\cdot)-\varphi_{s,R}(r,\cdot))*k\|_{L^2(\bR)}^2  \, dr\Bigg)^{p'/2} \les \Bigg( \int_0^T  \| \varphi_{t,R}(r,\cdot)-\varphi_{s,R}(r,\cdot)\|_{L^{q'}(\bR)}^2  \, dr\Bigg)^{p'/2},
\]
where $q' = \frac{2}{\beta}$. Applying \eqref{tsPhiDif} gives the desired bound, 
\begin{align*}
A & \les \Bigg( \int_0^T  R^{\frac{2}{q'}}(t-s)^2  \, dr\Bigg)^{p'/2}  \les R^{\beta p'/2}(t-s)^{p'}.
\end{align*}
Similarly for $B$,
\[
\big\|\big(\varphi_{t,R}(r,\cdot)-\varphi_{s,R}(r,\cdot)\big) * \k\big\|_{L^{p'}(\bR)}^{p'} \les \|\varphi_{t,R}(r,\cdot)-\varphi_{s,R}(r,\cdot)\|_{L^{q''}(\bR)}^{p'} \les R^{\frac{p'}{q''}} (t-s)^{p'}
\]
where
\[
q'' = 
\begin{cases}
p' & \text{if } \k \in L^1(\bR), \\
\frac{2p'}{2 + \alpha p'} & \text{if } \k = R_{1,\alpha/2}.
\end{cases}
\]

Since $\frac{p'}{q''} = 1 = \beta \leq \frac{\beta p'}{2}$ when $\k \in L^1(\bR)$ and  $\frac{p'}{q''}=1+\frac{\alpha p'}{2} \leq \frac{(\alpha+1)p'}{2}= \frac{\beta p'}{2}$ when $\k = R_{1,\alpha/2}$. It follows that for all $R\geq 1$,
\[
B \les R^{\beta p' / 2} (t-s)^{p'}.
\]
(Note that if $p' = 2$, then $\frac{p'}{q''} = \frac{\beta p'}{2}$ and the inequality is achieved for all $R > 0$.) The calculations for $A'$ and $B'$ follow similarily using \eqref{norm-phi},
\begin{align*}
A' & \leq \left( \int_s^t \| \varphi_{s,R}(r,\cdot)* \k \|^2_{L^2(\bR)} dr \right)^{p'/2} \les \left( \int_s^t \| \varphi_{t,R}(s,\cdot) \|_{L^{q'}(\bR)}^2 dr \right)^{p'/2}
\\
& \les \left( \int_0^T R^{\frac{2}{q'}} (t-s)^{2} dr \right)^{p'/2} \les R^{\beta p'/2} (t-s)^{p'}.
\end{align*}
Likewise for $B'$,
\begin{align*}
B' & \les \int_s^t \big\|\varphi_{t,R}(s,\cdot)\|_{L^{q''}(\bR)}^{p'}   dr \les
  \int_0^T R^{p'/q''} (t-s)^{p'} dr \les R^{\beta p'/2} (t-s)^{p'}.
\end{align*}
This concludes the proof of \eqref{incr-F}.
%% END FCLT

%% BEGIN Appendix
\vspace{30mm}

\appendix

\section{Auxiliary results from Malliavin calculus}
\label{appA}

In this appendix section, we include some auxiliary results from Malliavin calculus which were used in the sequel. 
\medskip

Let $(Z,\cZ,\fm)$ be a $\sigma$-finite measure space, and $N$ be a Poisson random measure on $Z$ of intensity $\fm$, defined on a probability space $(\Omega,\cF,\bP)$.  Let $\widehat{N}$ be the compensated version of $N$. 
We denote $\cH=L^2(Z,\cZ,\fm)$.

%\medskip

The following lemma was used in the proof of Proposition \ref{properties-v}.a).

\begin{lemma}
\label{Skor-lem}
Let $V \in L^2(\Omega;\cH)$ be such that for any $\xi \in Z$, $V(\xi)$ has the following chaos expansion (with respect to $N$):
\[
V(\xi)=\sum_{n\geq 0}I_n \big( g_n(\cdot,\xi)\big)
\]
where $g_0(\xi)=\bE[V(\xi)]$ and $g_n(\cdot,\xi) \in \cH^{\otimes n}$ may not be symmetric. Let $\delta$ be the Skorohod integral with respect to $\widehat{N}$. If $V \in {\rm Dom}(\delta)$, then
\[
\delta(V)=\sum_{n\geq 0}I_{n+1}(g_n).
\]
\end{lemma}

\begin{proof}
For any $\xi \in Z$, let $f_n(\cdot,\xi)$ be the symmetrization of $g_n(\cdot,\xi)$:
\begin{equation}
\label{def-fn}
f_n(\xi_1,\ldots,\xi_n,\xi)=\frac{1}{n!}\sum_{\rho \in S_n} g_n(\xi_{\rho(1)},\ldots,\xi_{\rho(n)},\xi).
\end{equation}

By Lemma 2.12 of \cite{BZ25}, 
$\delta(V)=\sum_{n\geq 0}I_{n+1}(\widetilde{f}_n)$,
where $\widetilde{f}_n$ is the symmetrization of $f_n$ in all $n+1$ variables:
\begin{align*}
\widetilde{f}_{n}(\xi_1,\ldots,\xi_n,\xi)&=\frac{1}{n+1}
\left[f_{n}(\xi_1,\ldots,\xi_n,\xi)+\sum_{i=1}^{n}
f_n(\xi_1,\ldots,\xi_{i-1},\xi,\xi_{i+1},\ldots,\xi_n,\xi_i) \right].
\end{align*}

Hence, it is enough to prove that $\widetilde{f}_n=\widetilde{g}_n$ for any $n\geq 1$, where 
\[
\widetilde{g}_n(\xi_1,\ldots,\xi_{n+1}) =\frac{1}{(n+1)!}\sum_{\sigma \in S_{n+1}} g_n(\xi_{\sigma(1)},\ldots,\xi_{\sigma(n+1)}).
\]

We calculate $\widetilde{f}_{n+1}(\xi_1,\ldots,\xi_{n+1})$. 
Using the fact that $f_n(\cdot,\xi_i)$ is symmetric, and definition \eqref{def-fn} of $f_n(\cdot,\xi)$, we have:
\begin{align*}
\widetilde{f}_{n}(\xi_1,\ldots,\xi_n,\xi_{n+1})
& =\frac{1}{n+1}\left[f_{n}(\xi_1,\ldots,\xi_n,\xi_{n+1})+\sum_{i=1}^{n}
f_n(\xi_1,\ldots,\xi_{i-1},\xi_{i+1},\ldots,\xi_n,\xi_{n+1},\xi_i) \right] \\
& =\frac{1}{(n+1)!}\left[\sum_{\rho \in S_n}g_{n}(\xi_{\rho(1)},\ldots,\xi_{\rho(n)},\xi_{n+1})+ \right.\\
& \qquad \qquad \quad \left. \sum_{i=1}^{n} \sum_{\rho \in \cP_{n,i}} g_n(\xi_{\rho(1)},\ldots,\xi_{\rho(i-1)},\xi_{\rho(i+1)}, \ldots, \xi_{\rho(n)},\xi_{\rho(n+1)},\xi_i)\right],
\end{align*}
where $\cP_{n,i}$ is the set of all permutations of $\{1,\ldots,i-1,i+1,\ldots,n,n+1\}$. The conclusion $\widetilde{f}_n=\widetilde{g}_n$ follows, using the fact that
\[
\sum_{\rho \in S_n}g_{n}(\xi_{\rho(1)},\ldots,\xi_{\rho(n)},\xi_{n+1})
=\sum_{\sigma \in S_{n+1}: \sigma(n+1)=n+1} g_n(\xi_{\sigma(1)},\ldots, \xi_{\sigma(n+1)}),
\]
and for any $i=1,\ldots,n$,
\[
\sum_{\rho \in \cP_{n,i}} g_n(\xi_{\rho(1)},\ldots,\xi_{\rho(i-1)},\xi_{\rho(i+1)}, \ldots, \xi_{\rho(n)},\xi_{\rho(n+1)},\xi_i)=\sum_{\sigma \in S_{n+1}: \sigma(n+1)=i} g_n(\xi_{\sigma(1)},\ldots, \xi_{\sigma(n+1)}).
\]
\end{proof}

Next, we present a result related to the Poisson product formula, which was used for the proof of Lemma \ref{prod-lem}. First, we recall some definitions. For any $f \in \cH^{\otimes n}$ and $g \in \cH^{\otimes m}$, we define the {\em modified contractions}, as follows:
\begin{itemize}
\item[(i)]
 $f\star^0_0 g = f\otimes g$ is the usual tensor product of $f$
and $g$;

\item[(ii)]
 for $1\leq k\leq n\wedge m$,   $f\star^0_k g$ is a  function
on $Z^{m+n-k}$ given by:
\begin{align*}
& (f \star^0_k g)(\gamma_1,\ldots,\gamma_k,x_1,\ldots,x_{n-k},x_1',\ldots,x_{n-k}')=\\
& \quad \quad \quad 
f(\gamma_1, \ldots, \gamma_k,  x_1, \ldots , x_{n-k})  g(\gamma_1, \ldots, \gamma_k, x_1', \ldots , x_{m-k}'),
\end{align*}

\item[(iii)]
for $1\leq \ell \leq k \leq n\wedge m$, $f\star^\ell_k g$ is  function
on $Z^{m+n-k-\ell}$, given by
\begin{align*}
& (f \star^{\ell}_k g) (\gamma_1,\ldots,\gamma_{k-\ell},x_1,\ldots,x_{n-k},x_1',\ldots,x_{n-k}')=\\
& \int_{Z^{\ell}} f(z_1,\ldots,z_{\ell},\gamma_1, ... , \gamma_{k-\ell},  x_1, \ldots, x_{n-k})  g(z_1,\ldots,z_{\ell},\gamma_1, \ldots, \gamma_k, x_1', \ldots , x_{m-k}')\fm(dz_1)\ldots \fm(dz_{\ell}).
\end{align*}
\end{itemize}

Note that  $f \star^{k}_k g$ coincides with the usual contraction $f \otimes_k g$, which appears in the product formula in the Gaussian case. For $\ell=1,\ldots,k-1$, 
$f \star^{\ell}_k g$ may not be well-defined. 

For the next result, we refer to Proposition 5 %(page 22) 
of \cite{Last16}, or relation (9.22) of \cite{NN18}.

\begin{proposition}[Poisson Product Formula]
\label{prop:prod}
Let $f \in \cH^{\otimes n}$ and $g \in \cH^{\otimes m}$ be symmetric functions such that
$f \star_{k}^{\ell} g\in \cH^{\otimes (m+n-k-\ell)}$
for any $k=1,\ldots, n\wedge m$ and $\ell=0,1,\ldots,k$.
Then,
\[
I_n(f) I_m(g)
= \sum_{k=0}^{n\wedge m}k! \binom{n}{k}\binom{m}{k}\sum_{\ell=0}^{k}\binom{k}{\ell}I_{n+m-k-\ell}(f \star_{k}^{\ell} g ).
\]
\end{proposition}

\begin{lemma}
\label{prod-lem2}
Let $f \in \cH^{\otimes n}$ and $g \in \cH^{\otimes m}$ (not necessarily symmetric).
If $\tilde{f} \star_{k}^{1} \tilde{g} =0$, then
\[
I_n(f) I_m(g)
= I_{n+m}(f\otimes g)+\sum_{k=1}^{n\wedge m}k! \binom{n}{k}\binom{m}{k}I_{n+m-k}(\tilde{f} \star_{k}^{0} \tilde{g}).
\]
\end{lemma}

\begin{proof}
By direct calculation, we see that $\tilde{f} \star_{k}^{1} \tilde{g} =0$ implies that $\tilde{f} \star_{k}^{\ell} \tilde{g} =0$ for all $\ell=2,\ldots,k$. Hence, all the terms in the product formula are zero, except those corresponding to $\ell=0$ and $\ell=1$:
\begin{align*}
I_n(f) I_m(g) &=I_n(\tilde{f}) I_m(\tilde{g})=I_{n+m}(\tilde{f} \otimes \tilde{g})+
\sum_{k=1}^{n\wedge m}k! \binom{n}{k}\binom{m}{k}I_{n+m-k}(\tilde{f} \star_{k}^{0} \tilde{g} ).
\end{align*}
Note that $I_{n+m}(\tilde{f} \otimes \tilde{g})=I_{n+m}(f \otimes g)$ since $\tilde{f} \otimes \tilde{g}$ and $f \otimes g$ have the same symmetrization.
\end{proof}

\section{Inequalities for Riesz potentials}
\label{appB}

In this section, we include some inequalities for Riesz potentials which were used for the proof of the QCLT in the case when $\k$ is the Riesz kernel.

We recall that the Riesz kernel of order $\alpha \in (0,d)$ is defined by 
\[
R_{d,\alpha}(x)=C_{d,\alpha}|x|^{-(d-\alpha)} \quad \mbox{with} \quad C_{d,\alpha}=\pi^{-\frac{d}{2}}2^{-\alpha} \frac{\Gamma(\frac{d-\alpha}{2})}{\Gamma(\frac{\alpha}{2})}.
\]
Its Fourier transform (in the sense of distributions) is $\cF R_{d,\alpha}(\xi)=|\xi|^{-\alpha}$, since
\[
\int_{\bR^d}\varphi(x)|x|^{-\alpha}dx=C_{d,\alpha} \int_{\bR^d}\cF \varphi(\xi)|\xi|^{-(d-\alpha)}d\xi \quad 
\mbox{for any $\varphi \in \cS(\bR^d)$},
\]
where $\cS(\bR^d)$ is the set of rapidly decreasing functions on $\bR^d$.

\medskip

For the next result, we refer to Theorem 1, p.119 of \cite{stein70}.

\begin{theorem}[Hardy-Littlewood-Sobolov inequality]
\label{HLS}
For any $\alpha \in (0,d)$ and $p>\frac{d}{d-\alpha}$, 
\[
\| \varphi *R_{d,\alpha}\|_{L^p(\bR^d)} \leq A_{d.\alpha,p}\|\varphi\|_{L^{q}(\bR^d)} \quad \mbox{for any $\varphi \in L^q(\bR^d)$},
\]
where $A_{d,\alpha,p}>0$ is a constant depending on $(d,\alpha,p)$, and $q$ is defined by:
\[
\frac{1}{q}=\frac{1}{p}+\frac{\alpha}{d}.
\]
(Condition $p>\frac{d}{d-\alpha}$ is equivalent to $q>1$.)
\end{theorem}

As a consequence of Theorem \ref{HLS}, we obtain the following lemma.

\begin{lemma}
\label{Hol-HLS}
For any $\alpha \in (0,d)$, $q \in (1,\frac{d}{\alpha})$, $\varphi \in L^q(\bR^d)$ and $\psi \in L^{q'}(\bR^d)$,
\[
\left|\int_{\bR^d}\int_{\bR^d}\varphi(x)\psi(y) R_{d,\alpha}(x-y) dxdy \right| \leq A_{d,\alpha,p} \|\varphi\|_{L^q(\bR^d)} \|\psi\|_{L^{q'}(\bR^d)},
\]
where $A_{d,\alpha,p}$ is the constant  from Theorem \ref{HLS}, and $(p,q')$ are defined by
\[
\frac{1}{p}+\frac{1}{q}=1 \quad \mbox{and} \quad \frac{1}{q}+\frac{1}{q'}=1+\frac{\alpha}{d}.
\]
\end{lemma}

\begin{proof}
We apply H\"older's inequality, letting $\frac{1}{p}+\frac{1}{q}=1$, and obtain that
\begin{align*}
\left|\int_{\bR^d}\int_{\bR^d}\varphi(x)\psi(y) R_{d,\alpha}(x-y) dxdy \right|
& =\left| \int_{\bR^d}\varphi(x) \big(\psi * R_{d,\alpha} \big) dx\right| \leq \|\varphi\|_{L^q(\bR^d)} \|\psi * R_{d,\alpha}\|_{L^p(\bR^d)}.
\end{align*}
Note that $q<\frac{d}{\alpha}$ is equivalent to $p>\frac{d}{d-\alpha}$.
Therefore, by Theorem \ref{HLS} 
\[
 \|\psi * R_{d,\alpha}\|_{L^p(\bR^d)} \leq A_{d,\alpha,p}  \|\psi \|_{L^{q'}(\bR^d)},
\]
where $\frac{1}{q'}=\frac{1}{p}+\frac{\alpha}{d}$. Finally, we note that $
1-\frac{1}{q}=\frac{1}{p}=\frac{1}{q'}-\frac{\alpha}{d}$, and hence, $\frac{1}{q}+\frac{1}{q'}=1+\frac{\alpha}{d}$.

\end{proof}

\begin{remark}
{\rm Applying Lemma \ref{Hol-HLS} with $d=1$ $\alpha=2H-1$ for some $H \in (\frac{1}{2},1)$, and  $q=q'=\frac{1}{H}$, we obtain that for any $\varphi,\psi \in L^{1/H}(\bR)$,
\[
\left|\int_{\bR}\int_{\bR}\varphi(x)\psi(y)|x-y|^{2H-2} dxdy \right| \leq C_{H} \|\varphi\|_{L^{1/H}(\bR)} \|\psi\|_{L^{1/H}(\bR)},
\]
where $C_H>0$ is a constant that depends on $H$. This inequality is useful for the stochastic analysis with respect to fractional Brownian motion (see e.g. relation (1.23) of \cite{nualart03}).
}
\end{remark}
%% END Appendix

%% BEGIN Biblio

%% END Biblio


\begin{thebibliography}{99}

\bibitem{B15} Balan, R.M. (2015). Integration with respect to L\'evy colored noise, with applications to SPDEs. {\em Stochastics} {\bf 87}, 363-381.
    
\bibitem{BJ25} Balan, R.M. and Jim\'enez, J.J. (2026). Moment estimates for solutions of SPDEs 
with L\'evy colored noise. {\em Stoch. \& PDEs.} To appear.
%Preprint arXiv:2504.21648.

    
\bibitem{BN16}  Balan, R.M. and Ndongo, C. (2016)  Intermittency for the wave equation with L\'evy white noise. {\em Stat. Probab. Letters} {\bf 109}, 214-223.    


\bibitem{BZ24} Balan, R.M. and Zheng, G. (2024). Hyperbolic Anderson model with L\'evy white noise: spatial ergodicty and fluctuations. {\em Trans. AMS} {\bf 377}, 4171-4221.
    
\bibitem{BZ25}  Balan, R.M. and Zheng, G. (2025). Central limit theorem for stochastoc nonlinear wave equation with pure-jump L\'evy white noise. {\em Electr. J. Probab.} Under review.
    %Preprint arXiv:2504.18672.
    
\bibitem{BNZ} Bola\~{n}os-Guerrero, R., Nualart, D. and Zheng, G. (2021). Averaging 2d stochastic wave equation. {\em Electr. J. Probab.}
{\bf 26}, paper no. 102, 32 pp.

\bibitem{CKNP-delta} Chen, L., Khoshnevisan, D., Nualart, D. and Pu, F. (2022).
Spatial ergodity and central limit theorems for parabolic Anderson model with delta initial condition.
{\em J. Funct. Anal.} {\bf 282}, 109290.

\bibitem{CKNP22} Chen, L., Khoshnevisan, D., Nualart, D. and Pu, F. (2022).
Poincar\'e inequality, and central limit theorems
for parabolic stochastic partial differential equations.
{\em Ann. Inst. Henri Poinca\'e: Prob. Stat.} {\bf 58}, 1052-1077.


\bibitem{dalang99} Dalang, R. C. (1999). Extending martingale
measure stochastic integral with applications to spatially
homogenous s.p.d.e.'s. {\em Electr. J. Probab.} {\bf 4}, no. 6,
1-29. Erratum in {\em Electr. J. Probab.} {\bf 4}, no. 6, 1-5.

\bibitem{DZ92} Da Prato, G. and Zabczyk, J. (1992).
{\em Stochastic Equations in Infinite Dimensions}, Cambridge
University Press.


\bibitem{DNZ20} Delgato-Vences, F., Nualart, D. and Zheng, G. (2020). A Central Limit Theorem for the stochastic wave equation with fractional noise. {\em Ann. Inst. Henri Poincar\'e: Prob. Stat.} {\bf 56}, 3020-3042.
    
\bibitem{hairer14} Hairer, M. (2014). A theory of regularity structures. {\em Invent. Math.} {\bf 198}, 269-504.    
    
\bibitem{HNV20} Huang, J., Nualart, D. and Viitasaari, L. (2020). A central limit theorem for the stochastic heat equation. {\em Stoch. Proc. Appl.} {\bf 130}, %{no. 12}
    7170-7184.

\bibitem{HNVZ} Huang, J., Nualart, D., Viitasaari L. and Zheng, G. (2020). Gaussian fluctuations for the stochastic
heat equation with colored noise. {\em Stoch. PDE: Anal. Comp.} {\bf 8}, 402-421.


\bibitem{kallenberg02} Kallenberg, O. (2002). {\em  Foundations of Modern Probability}. Second Edition. Springer, New York.    
    

\bibitem{Last16} Last, G. (2016). Stochastic analysis for Poisson processes.
In:  ``Stochastic analysis for Poisson point processes'', Peccati, G. and Reitzner, N. eds.
%Mathematics, Statistics, Finance and Economics, chapter 1, pages 
1--36, Bocconi Springer, Bocconi Univ. Press.

\bibitem{LPS16}
Last, G.,  Peccati, G. and  Schulte, M. (2016).
Normal approximation on Poisson spaces:
Mehler's formula, second order Poincar\'e inequalities and stabilization.
{\em Probab. Th. Rel. Fields} {\bf 165}, 667-723.

\bibitem{lyons98} Lyons, T.J. (1998). Differential equations driven by rough signals. {\em Rev. Mat. Iberoamericana} {\bf 14}, 215–310. 

\bibitem{nualart03} Nualart, D. (2003). Stochastic integration with respect to
 fractional Brownian motion and applications. {\em Contem. Math.}
 {\bf 336}, 3-39.
 

\bibitem{NN18} Nualart, D. and Nualart, E. (2018). {\em Introduction to Malliavin calculus}.
Cambridge University Press, Cambridge. 
 
 \bibitem{NSZ20} Nualart, D., Song, X. and Zheng, G. (2020). Spatial averages for the Parabolic Anderson model driven by rough noise. {\em  ALEA Lat. Am. J. Probab. Math. Stat.} {\bf 18}, %{no. 1}, 
    907-943.

\bibitem{NXZ22} Nualart, D., Xia, P. and Zheng, G. (2022).  Quantitative central limit theorems for the parabolic Anderson model driven by colored noises. {\em Electr. J. Probab.} {\bf 27}, paper no. 120, 43 pp.

\bibitem{NZ20-1} Nualart, D. and Zheng, G. (2020). Averaging Gaussian functionals. {\em Electr. J. Probab.} {\bf 25}, paper no. 48, pp. 54.

\bibitem{NZ20-2} Nualart, D. and Zheng, G. (2020). Spatial ergodicity of stochastic wave equations in dimensions 1, 2 and 3. {\em Electr. Comm. Probab.} {\bf 25}, paper no. 81, pp 11.


\bibitem{NZ22} Nualart, D. and Zheng, G. (2022). Central limit theorems for stochastic wave equations in dimensions one and two. {\em Stoch. PDE: Anal. \& Comput.} {\bf 22}, 392-418.
    
\bibitem{sanz05} Sanz-Sol\'e, M. (2005). {\em Malliavin Calculus with Applications to Stochastic Partial Differential Equations}. EPFL Press, Lausanne.    
    
\bibitem{stein70} Stein, E.M. (1970). {\em Singular Integrals and Differentiability Properties of Functions}. Princeton University Press, Princeton.
    
\bibitem{trauthwein25} Trauthwein, T. (2025). Quantitative CLTs on the Poisson space via Skorohod estimates and $p$-Poincar\'e inequalities. {\em Ann. Appl. Probab.} {\bf 35}, 1716-1754.

\bibitem{walsh86} Walsh, J. B. (1986). An introduction to stochastic
partial differential equations. {\em Ecole d'Et\'{e} de
Probabilit\'{e}s de Saint-Flour XIV. Lecture Notes in Math.} {\bf
1180}, 265-439. Springer, Berlin.

\end{thebibliography}
\end{document}